\documentclass[smallextended]{svjour3}       
\pdfoutput=1
\smartqed  
\usepackage{graphicx}
\usepackage{amsmath}
\usepackage{amssymb,bbold}
\usepackage{stmaryrd}
\usepackage[lofdepth,lotdepth]{subfig}
\newcommand{\E}{\mathbb{E}}
\newcommand{\Var}{\operatorname{Var}}
\renewcommand{\P}{\mathbb{P}}
\DeclareMathOperator{\supp}{Supp}
\journalname{Statistics and Computing}

\begin{document}

\title{Conditional quantile sequential estimation for stochastic codes}

\titlerunning{Conditional quantile and stochastic codes}       

\author{T.  Labopin-Richard   \and
        F. Gamboa \and A. Garivier \and
        J. Stenger}


\institute{T.  Labopin-Richard, F. Gamboa and J. Stenger \at
              Institut de math\'ematique de Toulouse, Universit\'e Paul Sabatier \\
118 Route de Narbonne\\
31062 Toulouse, France\\
\email{fabrice.gamboa@math.univ-toulouse.fr}           
           \and
A. Garivier \at
Unit\'e de Math\'ematiques Pures et Appliqu\'ees, Laboratoire de l'Informatique du Parall\'elisme\\
\'Ecole Normale Sup\'erieure de Lyon, Universit\'e de Lyon\\
46, allée d'Italie, Lyon, France
\email{aurelien.garivier@ens-lyon.fr}  \\
\and 
J. Stenger \at
EDF R$\&$D\\
6 Quai Watier\\
78400 Chatou, France
\email{jerome.stenger@edf.fr}}

\date{Received: date / Accepted: date}

\maketitle

\begin{abstract}
We propose and analyze an algorithm for the sequential estimation of a conditional quantile in the context of real stochastic codes with vector-valued inputs. Our algorithm is based on $k$-nearest neighbors smoothing within a Robbins-Monro estimator. We discuss the convergence of the algorithm under some conditions on the stochastic code. We provide non-asymptotic rates of convergence of the mean squared error and we discuss the tuning of the algorithm's parameters.
\keywords{Stochastic code \and Conditional quantile \and Robbins-Monro stochastic algorithm \and $k$-nearest neighbors method}
 \subclass{62L12 \and 62L20 \and 62G32}
\end{abstract}

\section{Introduction}

Computer code experiments have encountered, in the last decades, a growing interest among statisticians in several fields (see \cite{code2}, \cite{code1,code3,oakley,jala2,bect} and references therein). In the absence of noise, a numerical black box $g:\mathbb{R}^d\to\mathbb{R}$ maps an \emph{input vector} $X$ to a numerical output $Y=g(X)\in\mathbb{R}$.
When the black box does include some randomness, the code is called \emph{stochastic} and the model is as follows: a random vector $\varepsilon\in\mathbb{R}^m$, called \emph{random seed}, models the stochasticity of the function, while $X$ is a random vector. The random seed and the input are assumed to be stochastically independent. The map $g$ (which satisfies some regularity assumption specified below)  is defined on $\mathbb{R}^d \times \mathbb{R}^m$ and outputs 
\begin{equation}
\label{box}
Y=g(X, \varepsilon)\;,
\end{equation}
hence yielding possibly different values for the same input $X$.
One observes a sample of pairs $(X, Y)$, without having access to the details of $g$. In the context of computer experiments, those observations are often expensive (for example when $g$ has a high computational complexity) and one aims at learning rapidly some properties of interest on $g$.

We focus in this work on the estimation of the conditional quantile of the output $Y$ given the input $X$. For a given level $\alpha \in [1/2, 1)$ and for every possible input $x\in\mathbb{R}^d$, the target is
\[\theta^*(x):=q_{\alpha} \big(  g(x, \varepsilon) \big) \,,\quad x \in \mathbb{R}^d \;,\]
where $q_{\alpha}(Z) :=  F_Z^{-1}(\alpha)$ is the quantile of level $\alpha$ of the random variable $Z$ and $F_Z^{-1}(u):= \inf \{x : F_Z(x) \geq u \}$ is the generalized inverse of the cumulative distribution function of $Z$. Notice that we restrict $\alpha \in [1/2, 1)$ as the case $\alpha \in (0, 1/2]$ can be tackled in the same way considering $-Z$. Our goal is to estimate the conditional quantile for different values of $x$ at the same time. 

\subsection*{The algorithm}

For a fixed value of $x$, there are several well-known procedures to estimate the quantile $\theta^*(x)$.
Given a sample $(Y_i^x)_{i=1 \dots n}$ of $Y^x:=g(x, \varepsilon)$, the empirical quantile is a solution. For a sequential estimation, one may use a Robbins Monro~\cite{rob} estimator. This method permits to iteratively approximate the zero of a function $h: \mathbb{R} \to \mathbb{R}$ by a sequence of estimators defined by induction: $\theta_0\in\mathbb{R}^d$ and for all $n\geq 0$,
\[\theta_{n+1} =  \theta_n - \gamma_{n+1} H(\theta_n, Z_{n+1})\; .\]
Here, $(\gamma_n)$ is the learning rate (a deterministic step-size sequence), $(Z_n)$ is an i.i.d sample of observations, and $H$ is a noisy version of $h$. Denoting $\mathcal{F}_n:= \sigma (Z_1, \dots Z_n)$ the sigma-field induced by the observations, $H$ is such that \[\E\big(H(\theta_n, Z_{n+1})| \mathcal{F}_n\big) = h(\theta_n) \;.\]
Classical conditions for the the choice of the step sizes $(\gamma_n)$ are
$$\displaystyle\sum_{n} \gamma_n^2 < \infty, \text{ and } \displaystyle\sum_{n} \gamma_n= \infty\;.$$ 
These conditions ensure the convergence of the estimates under weak assumptions. For example, convergence in mean squared is studied in~\cite{rob}, almost sure consistency is considered in~\cite{blum,schrek}, asymptotic rate of convergence are given in~\cite{fab,rup,sac}, while large deviations principles are investigated in~\cite{woodroofe}.
There has been a recent interest on non-asymptotic results. Risk bounds under Gaussian concentration assumption (see~\cite{meno}) and finite time bounds on the mean squared error under strong convexity assumptions (see~\cite{moul,schrek} and references therein), have been given.
Quantile estimation corresponds to the choice $h:t \mapsto F(t)-\alpha$, where $F$ is the cumulative distribution function of the target distribution. One can show that the estimator 
\begin{equation}
\label{algoquantile}
\left \lbrace
\begin{aligned}
&\theta_0   \in \mathbb{R} \\
&\theta_{n+1} =  \theta_n - \gamma_{n+1} \left( \mathbb{1}_{Z_{n+1} \leq \theta_{n}} - \alpha \right) \\
\end{aligned}
\right. 
\end{equation}
is consistent and asymptotically Gaussian (see~\cite{Duflo} chapters 1 and 2 for proofs and details). It is important to remind, however, that the lack of strong convexity prevents most non-asymptotic results to be applied directly, except when the density is lower-bounded. We nevertheless mention that Godichon et al. prove in \cite{cenac,godichon} such non-asymptotic results for the adaptation of algorithm (\ref{algoquantile}) to the case where $Z$ is a random variable on an Hilbert space of dimension higher than 2.

Of course, unless $x$ can take a small number of different values, it is not possible to use this algorithm with a sample of $Y^x$ for each possible input value $x$. 
Even more, when the code has a high computational complexity, the overall number of observations (all values of $x$ included) must remain small, and we need an algorithm using only one limited sample $(X_i, Y_i)_{i=1 \dots n}$ of $(X, Y)$. 
Then, the problem is more difficult. For each value of~$x$, we need to estimate quantile of the conditional distribution given $x$ using a \emph{biased} sample. To address this issue, we propose to embed Algorithm~\eqref{algoquantile} into a non-parametric estimation procedure. For a fixed input $x$, the new algorithm only takes into account the pairs $(X_i,Y_i)$ for which the input $X_i$ is close to $x$, and thus (presumably) the law of $Y_i$ close to that of $Y^x$. To set up this idea, we use the $k$-nearest neighbors method, introducing the sequential estimator:

\begin{equation}
\label{algovoisin}
\left \lbrace
\begin{aligned}
&\theta_0(x)   \in \mathbb{R} \\
&\theta_{n+1}(x) =  \theta_n(x) - \gamma_{n+1}   \left(\mathbb{1}_{Y_{n+1} \leq \theta_n(x)} - \alpha\right)\mathbb{1}_{X_{n+1} \in kNN_{n+1}(x)} \;,
\end{aligned}
\right.
\end{equation}
where

\begin{itemize}

\item[$\bullet$] $kNN_n(x)$ is the subset of $\{X_1,\dots,X_n\}$ made of the $k_n$ nearest neighbors of $x$ for the euclidean norm on $\mathbb{R}^d$. Denoting by $||X-x||_{(i,n)}$ the $i$-th statistic order of a sample $\big(||X_i-x||\big)_{i=1 \dots n}$ of size $n$, we have

$$\big\{X_{n+1} \in kNN_{n+1}(x) \big\}= \big\{ ||X_{n+1}-x|| \leq ||X-x||_{(k_{n+1}, n)}\big\}\; .$$

In this work, we discuss choices of the form $k_n= \lfloor n^{\beta} \rfloor$ for $0<\beta<1$, $n \in \mathbb{N}^*$.

\item[$\bullet$] $(\gamma_n)$ is the deterministic steps sequence. We focus here on the choice $\gamma_n= n^{-\gamma}$ with $0<\gamma \leq 1$. 
\end{itemize}

The $k$-nearest neighbors method of localization first appears in \cite{sto1,sto2} for the estimation of conditional expectations. In \cite{bata2}, Bhattacharya et al. apply it to the (non-recursive) estimation of the conditional quantile function for real-valued inputs.
Regarding the computational cost of the algorithm (\ref{algovoisin}), 
naive implementations of the search for nearest neighbors require $O(n)$ operations at round $n$, which means that the overall complexity is quadratic. However, the smart use of quad-trees (a hierarchical partition of space) permits to reduce the cost of an iteration to $O(\log(n))$, and in practice the algorithm has almost a linear complexity.

Remark that if the number of neighbors $k_n$ is small, then few observations are used and the estimation is highly noisy; on the other hand, if $k_n$ is large, then values of $Y_i$ may be used that have a distribution significantly different from the target. The challenge is thus to tune $k_n$ so as to reach an optimal balance between bias and variance.

In this work, this tuning is combined with the choice of the learning rate. The main objective of this work is to optimize the choice of the two parameters $\beta$ and $\gamma$ of Algorithm~\eqref{algovoisin} that monitor the learning rate $\gamma_n$ and the number of neighbors $k_n$.  The paper is organized as follows: Section~\ref{sec:main} deals with the stability, and with the almost sure convergence of the algorithm. Furthermore, it contains the main result of our paper: a non-asymptotic inequality on the mean squared error from which  an optimal choice of parameters is derived. In Section~\ref{sec:simu}, we present some numerical simulations to illustrate our results. The technical points of the proofs are deferred to Appendix~\ref{sec:proofs}, while Appendix~\ref{Apx2} summarizes the notation and constants used in this paper. 

\section{Main results}\label{sec:main}
After giving some notation and technical assumptions, we explain in this section how to tune the parameters of the algorithm. We also provide conditions allowing theoretical guarantees of convergence.

\subsection{Notation}
The constants appearing in the sequel are of three different types:
\begin{itemize}
\item[1)] $(L, U)$ denote lower- and upper bounds for the support of random variables. They are indexed by the names of those variables;
\item[2)] $(N_i)_{i \in \mathbb{N}^*}$ are integers denoting the first ranks after which some properties hold;
\item[3)] $(C_i)_{i \in \mathbb{N}^*}$ are positive real numbers used for other purposes.
\end{itemize}
Without further precision, constants of type 2) and 3) only depend on the model, that is, on $g$ and on the distribution of $(\varepsilon, X)$. Further, we denote by $C_i(u)$ or $N_i(u)$, $u \in \mathcal{P}(\{\alpha, x, d\})$ (the power set of a $\{\alpha, x, d\}$), constants depending on the model, on the probability level $\alpha$, on the point $x$ and on the dimension $d$. The values of all the constants are summarized in Appendix~\ref{Apx2}.

\vspace{0.3cm}

For any random variable $Z$, we denote by $F_Z$ its cumulative distribution function. 
We denote by $\mathcal{B}_x$ the set of the balls of $\mathbb{R}^d$ centred at $x$. For $B \in \mathcal{B}_x$, we denote by $r_B$ its radius and for $r_B>0$, we call $Y^B$ a random variable with distribution $\mathcal{L}(Y | X \in B)$.

\begin{remark} If the pair $(X, Y)$ has a density $f_{(X, Y)}$ with respect to Lebesgue measure and if the marginal density $f_X(x)$ is positive, then the density of $\mathcal{L}(Y | X=x)$ is
\[f_{Y |X=x}= \frac{f_{(X, Y)}(x, .)}{f_X(x)}\;,\]
and when $B=\{x\}$,  $$Y^B \overset{\mathcal{L}}{=} Y^x =  g(x, \varepsilon) \sim \mathcal{L}(Y| X=x)\;.$$
\end{remark}

\subsection{Almost sure convergence}

In order to prove the convergence of our algorithm, we make two assumptions. The first one, a continuity assumption on the code, can hardly be avoided for our $k$-nearest neighbors to be valid. The second one is convenient for the simplicity of the analysis.
\vspace{0.3cm}

\textbf{Assumption A1} For all $x$ in the support of $X$ (that we will denote $\supp(X)$ in the sequel), there exists a constant $M(x)$ such that the following inequality holds :

\[\forall B \in \mathcal{B}_x, \, \, \forall t \in \mathbb{R}, \, \, |F_{Y^B}(t) - F_{Y^x}(t)| \leq M(x) \,r_B\;.\]
In words, we assume that the stochastic code is sufficiently smooth. The law of two responses corresponding to two different but close inputs are not completely different. The assumption is clearly required, since we want to approximate the law $\mathcal{L}\left(Y|X=x\right)$ by the law $\mathcal{L}(Y| X \in kNN_n(x))$.

\begin{remark}
If we consider random vector supported by $\mathbb{R}^d \times \mathbb{R}$, we can show that Assumption \textbf{A1} holds, for example, as soon as $(X, Y)$ has a regular density with respect to Lebesgue measure. In all cases, it is easier to prove this assumption when the couple $(X, Y)$ has a density: see Subsection 3.1 for an example. 
\end{remark}
\vspace{0.3cm}

\textbf{Assumption A2} The law of $X$ has a density with respect to Lebesgue measure, and this density is lower-bounded by a constant $C_{input}>0$ on $\supp(X)$.

\vspace{0.3cm}

\noindent This hypothesis implies in particular that the law of $X$ has a compact support of volume at most $\frac{1}{C_{input}}$. This kind of assumptions is usual in $k$-nearest neighbors context (see for example \cite{gadat}). 
The following theorem studies the almost sure convergence of our algorithm.

\vspace{0.3cm}

\begin{theorem}
\label{cvas}
Let $x$ and $\alpha$ be fixed. Under Assumptions $\textbf{A1}$ and $\textbf{A2}$, Algorithm (\ref{algovoisin}) is almost surely convergent whenever $\frac{1}{2}< \gamma \leq \beta < 1$.
\end{theorem}

\vspace{0.3cm}

%
%
%
%
%
%
%
%
%
%

\noindent\textbf{Comments on parameters.} In the Theorem \ref{cvas}, we assume that $0< \beta< 1$. This means that the number of neighbors goes to $+ \infty$ and  $||X-x||_{(k_n, n)} \rightarrow 0$, as $P(X\in [x-\xi,x+\xi])>0,\;\forall \xi\geq 0.$ Obviously, the "localization" condition $k_n/n\to 0$ requires $\beta<1$: it is quantitatively exploited in~Lemma \ref{dev}. The condition $\beta \geq \gamma$ can be informally understood in this way. When considering Algorithm (\ref{algoquantile}), we deal with the \textit{global learning rate} $\gamma_n=n^{-\gamma}$. 
In Algorithm (\ref{algovoisin}), since for a fixed input $x$, there is not an update at each step $n$, one may define an \textit{effective learning rate} $\gamma_{k_n}$ as follows. At step $k$, 
$\theta_k(x)$ has a probability of $\P\big(X_{n+1} \in kNN_{n+1}(x)\big)\approx k^{\beta}/k$ to be updated (see Lemma \ref{Pn}). Up to step $n$, the estimator is thus updated a number of times approximately equal to $$N=\displaystyle\sum_{k \leq n} k^{\beta-1} = O\big( n^{\beta} \big)\;.$$ Thus, one has to wait on average up to step $O\big(n^{\frac{1}{\beta}})$ in order to reach  $n$ updates. Hence, on average, the estimator of the quantile at $x$ evolves with Robbins-Monro iterations roughly equivalent to
$$\theta_{k_n}(x)=\theta_{k_n-1}(x) + \gamma_{k_n}\left( \mathbb{1}_{Y_{k_n}\leq \theta_{k_n}(x)} - \alpha \right)\;,$$
with the learning rate 
$$\gamma_{k_n}=\frac{1}{\left(n^{\frac{1}{\beta}}\right)^{\gamma}}= \frac{1}{n^{\frac{\gamma}{\beta}}}\;.$$
This is a well-known fact that this algorithm has a \textit{good} behaviour if, and only if, the sum
$$\displaystyle\sum_{n} \gamma_{k_n}= \displaystyle\sum_{n} \frac{1}{n^{\frac{\gamma}{\beta}}}\; ,$$
is divergent. That is if, and only if $\beta \geq \gamma$. At last, the condition $\frac{1}{2}<\gamma\leq 1$ is a classical assumption on the Robbins Monro algorithm to be consistent (see for example in \cite{rob}). Here, we restrict the condition to $\gamma<1$ because we need $1>\beta \geq \gamma$. The proof of Theorem~\ref{cvas}, in Appendix~\ref{sec:proofs}, gives rigorous foundations to this heuristic discussion.

\subsection{Rate of convergence of the mean squared error}

We now study the rate of convergence of the mean squared error $a_n(x):= \E \left( \left( \theta_n(x)- \theta^*(x) \right)^2 \right)$. Two rather technical assumptions are required. 
\vspace{0.3cm}

\textbf{Assumption A3} The code function $g$ takes its values in a compact interval $[L_Y, U_Y]$. 

\vspace{0.3cm}

\noindent Under Assumption \textbf{A3}, Lemma \ref{borne} (see Appendix~\ref{sec:proofs}) explains why if $\beta \geq \gamma$, then $\theta_n(x)$ is almost-surely bounded in an fixed interval $[L_{\theta_n}, U_{\theta_n}]$,  and that $\left|\theta_n(x)- \theta^*(x)\right|$  is upper-bounded by \[\sqrt{C_1}:= \max \left(U_Y-L_Y+(1 - \alpha),\; U_Y + \alpha -L_Y \right)=U_Y-L_Y+\alpha\;.\]

\vspace{0.3cm}

\textbf{Assumption A4} For all $x$, the law of $g(x, \varepsilon)$ has a density with respect to Lebesgue measure which is lower-bounded by a constant $C_g(x) >0$ on its support.

\begin{lemma}
\label{A4}
Denoting $C_2(x, \alpha):= \min \left(C_g(x), \frac{1-\alpha}{U_Y+\alpha-L_Y}, \right)$, it holds under Assumption $\textbf{A3}$ and $\textbf{A4}$ that for all $n$ in $\mathbb{N}^*$
\begin{equation}
\begin{aligned}
\label{galere}
\big[ F_{Y^x}(\theta_n(x)) - F_{Y^x}(\theta^*(x)) \big] \big[\theta_n(x) - \theta^*(x) \big] \geq C_2(x, \alpha) \big[\theta_n(x) - \theta^*(x) \big]^2.
\end{aligned}
\end{equation}
\end{lemma}

\begin{proof}
When $\theta_n(x) \in [L_Y, U_Y]$, it is obvious that Inequality (\ref{galere}) holds for $C_2:=C_g(x)$. When $\theta_n(x) \in [L_{\theta_n}, L_Y]$, we have 
\[L_{\theta_n} \leq \theta_n(x) \leq L_Y \leq \theta^*(x)\;,\]
and then $F_{Y^x}(\theta_n(x))=0$. Thus,
\begin{equation*}
\begin{aligned}
(\theta_n(x) - \theta^*(x))(F_{Y^x}(\theta_n(x)) - F_{Y^x}(\theta^*(x))) &= (\theta_n(x) - \theta^* (x))^2 \frac{(0- \alpha)}{\theta_n(x)-\theta^*(x)} \\
&= (\theta_n(x)- \theta^*(x))^2 \frac{\alpha}{\theta^*(x)-\theta_n(x)}\\
&\geq (\theta_n(x)- \theta^*(x))^2 \frac{\alpha}{U_Y+ \alpha - L_Y}\\
&\geq (\theta_n(x)- \theta^*(x))^2 \frac{1-\alpha}{U_Y+ \alpha - L_Y}\\
&\geq C_2(x, \alpha)(\theta_n(x)- \theta^*(x))^2\;.\\
\end{aligned}
\end{equation*}
\vspace{0.3cm}
The last case $\theta_n(x)\in[U_Y, U_{\theta_n}]$ can be treated similarly, using that $C_2(x, \alpha) \leq \frac{1-\alpha}{U_Y+\alpha -L_Y}$.
\end{proof}

\vspace{0.3cm}

This lemma is useful to deal with non-asymptotic inequality for the mean squared error. It is the substitute of the strong convexity assumption on the function to minimize, which is often made in the analysis of Robins-Monro stochastic approximation (see for example in\cite{moul}) but which does not hold for quantile estimation.

\begin{theorem}
\label{ineg}
Under hypothesis $\textbf{A1}$, $\textbf{A2}$, $\textbf{A3}$ and $\textbf{A4}$, the mean squared error $a_n(x)$ of the algorithm (\ref{algovoisin}) satisfies the following inequality : $\forall (\gamma, \beta, \zeta)$ such that $0<\gamma\leq \beta <1$ and $1>\zeta >1- \beta$, $\forall n > N_0 := 2^{\frac{1}{\zeta-(1-\beta)}}$, 

\begin{equation*}
\begin{aligned}
a_n(x) & \leq \exp \left( -2 C_2(x, \alpha) (\kappa_n- \kappa_{N_0})\right) C_1 
 +  \displaystyle\sum_{k=N_0+1}^{n} \exp \left( -2 C_2(x, \alpha)\left( \kappa_n - \kappa_k \right) \right) d_k \\
 & + C_1 \exp \left( - \frac{3n^{1- \zeta}}{8} \right)\;,\\
\end{aligned}
\end{equation*}
where for $j \in \mathbb{N}^*$, $\kappa_j= \displaystyle\sum_{i=1}^j i^{-\zeta - \gamma}$ and
\[d_n = C_1 \exp \left( - \frac{3n^{1- \zeta}}{8} \right) +2\sqrt{C_1}M(x) C_3(d) \gamma_{n} \left(\frac{k_n}{n}\right)^{\frac{1}{d}+1} + \gamma_{n}^2\frac{k_n}{n}\;.\]
Here, $C_3(d)>0$ is a constant depending on the dimension $d$ and on the distribution of $X$ (as recalled in Apprendix~\ref{Apx2}).
\end{theorem}

\vspace{0.3cm}

\textbf{Sketch of proof :} Following~\cite{moul}, the idea of the proof is to establish a recursive inequality on $a_n(x)$, that is for $n \geq N_0$,
\[a_{n+1}(x) \leq a_n(x) (1- c_{n+1}) + d_{n+1}\]
where for all $n \in \mathbb{N}^*$, $0 < c_n<1$ and $d_n >0$. We use the technical Lemma~\ref{Guillaume}. In this purpose we begin by expanding the square

\begin{equation*}
\begin{aligned}
(\theta_{n+1}(x)- \theta^*(x))^2 & = (\theta_n(x)- \theta^*(x))^2\\ & +  \gamma_{n+1}^2 \left[ (1 - 2 \alpha) \mathbb{1}_{Y_{n+1} \leq \theta_n(x)} + \alpha^2 \right] \mathbb{1}_{X_{n+1} \in kNN_{n+1}(x)} \\
& - 2 \gamma_{n+1}(\theta_n(x) - \theta^*(x))\left( \mathbb{1}_{Y_{n+1} \leq \theta_n(x)} - \alpha \right) \mathbb{1}_{X_{n+1} \in kNN_{n+1}(x)}\;. \\
\end{aligned}
\end{equation*}
Taking the expectation conditionally to  $\mathcal{F}_n := \sigma(X_1, \dots, X_n, Y_1, \dots, Y_n)$, using $(1 - 2 \alpha) \mathbb{1}_{Y_{n+1} \leq \theta_n(x)}$ $+ \alpha^2 \leq 1$ and $\alpha = F_{Y^x}(\theta^*(x))$, we obtain thanks to the Bayes formula that
\begin{equation}
\label{remplace}
\begin{aligned}
\E_n \left( \left(\theta_{n+1}(x)- \theta^*(x)\right)^2 \right) & \leq \E_n \left( \left(\theta_{n}(x)- \theta^*(x) \right)^2 \right) + \gamma_{n+1}^2 P_{n} \\
& -  2 \gamma_{n+1} \left(\theta_n(x)-\theta^*(x) \right)\\
& \times P_{n} \left[ F_{Y^{B_n^{k_{n+1}}(x)}}(\theta_n(x)) - F_{Y^x}(\theta^*(x))  \right] \;,\\
\end{aligned}
\end{equation}
where $P_{n}:=\P_n \left(X_{n+1} \in kNN_{n+1}(x) \right)$ 
and $B_n^{k_{n+1}}(x)$ is the ball of $\mathbb{R}^d$ centred in $x$ and of radius $||X-x||_{(k_{n+1},n)}$. We rewrite this inequality so as to highlight the presence of two different contributions to the risk: 

\begin{itemize}
\item[1)] First, the quantity $F_{Y^{B_n^{k_{n+1}}(x)}}(\theta_n(x)) - F_{Y^x}(\theta_n(x))$ represents the \textit{bias error} (due to the use of a biased sample of $F_{Y^x}$). Using Assumption~\textbf{A1}, it can be upper-bounded as
\[|F_{Y^{B_n^{k_{n+1}}(x)}}(\theta_n(x)) - F_{Y^x}(\theta_n(x))|  \leq  M(x) ||X-x||_{(k_{n+1}, n)}\;.\]
Moreover, by Assumption~\textbf{A3}, $|\theta_n(x)- \theta^*(x)| \leq \sqrt{C_1}$. Thus,
\begin{multline*}
\left|2 \gamma_{n+1} (\theta_n(x)- \theta^*(x))  P_{n} \left[ F_{Y^{B_n^{k_{n+1}}(x)}}(\theta_n(x)) - F_{Y^x}(\theta_n(x)) \right]\right| \\\leq 2 \gamma_{n+1}\sqrt{C_1}M(x) P_{n} ||X-x||_{(k_{n+1}, n)}\;.
\end{multline*}

\item[2)] The second quantity, $F_{Y^x}(\theta_n(x)) - F_{Y^x}(\theta^*(x))$ represents the \textit{on-line learning error} (due to the use of a stochastic optimization algorithm). Thanks to Assumption \textbf{A4} we obtain
\[\left( \theta_n(x) - \theta^*(x) \right)  \left[F_{Y^x}(\theta_n(x))- F_{Y^x}(\theta^*(x)) \right] 
 \geq C_2(x, \alpha) \left[ \theta_n(x) - \theta^*(x) \right]^2\;.\]
\end{itemize}
Taking the expectation in Inequality~\eqref{remplace} yields
\begin{equation*}
\begin{aligned}
a_{n+1}(x) & \leq a_n(x) - 2 \gamma_{n+1} C_2(x, \alpha) \E \left[ (\theta_n(x) - \theta^*(x))^2 P_{n} \right] + \gamma_{n+1}^2 \E(P_{n}) \\
&+ 2 \gamma_{n+1}M(x) \sqrt{C_1} \E(||X-x||_{(k_{n+1}, n)} P_{n})\;.\\
\end{aligned}
\end{equation*}

This inequality reveals a problem : thanks to Lemmas \ref{Pn} and \ref{unif} (and thus thanks to assumption \textbf{A2}) we can deal with the last two terms, but we are not able to evaluate directly $\E \left[ (\theta_n(x) - \theta^*(x))^2 P_{n} \right]$.
In order to solve this problem, we use a truncation parameter $\zeta_{n}$. Instead of writing a recursive inequality on $a_n(x)$ we write such inequality with the quantity $b_n(x):= \E \left[ \left(\theta_n(x) - \theta^*(x)\right)^2 \mathbb{1}_{P_{n} > \zeta_{n}} \right]$. Choosing $\zeta_n=n^{-\zeta}$, we have to tune another parameter but thanks to $\mathbf{A3}$ and deviation inequalities recalled in Lemma~\ref{dev}, we obtain a recursive inequality on $a_n(x)$ from the one on $b_n(x)$, for $n \geq N_0$. 

\vspace{0.3cm}

\textbf{Comments on the parameters.} We choose $0 < \beta <1$ for the same reasons as in Theorem \ref{cvas}. Regarding $\gamma$, the inequality is true for all $0<\gamma \leq \beta$ (which is unusual, as you can see in \cite{godichon} for example). We will nevertheless see in the sequel that this is not because the inequality is true that $a_n(x)$ converges to 0. We will discuss later \textit{good} choices for $(\gamma, \beta)$. 

\vspace{0.3cm}

\textbf{Compromise between the two errors.} This analysis emphasizes the necessity of a compromise on $\beta$ to deal with the two previous errors. Indeed,

\begin{itemize}
\item[$\bullet$] the \textit{bias error} gives the term \[\exp \left( -2 C_2(x, \alpha)(x)\displaystyle\sum_{k=N_0+1}^{n} \frac{1}{k^{\zeta+\gamma}} \right)\;,\] of the inequality. This term decreases to 0 if and only if $\gamma+ \zeta < 1$ which implies $\beta > \gamma$. It suggests that $\beta$ should not be chosen too small. 

\item[$\bullet$] the \textit{on-line learning error} gives the term $\left( k_n/n \right)^{1/d+1}=n^{-(1-\beta)(1+1/d)}$ in the remainder. For the remainder to decrease to 0 with the faster rate, we then need $\beta$ to be as small as possible compared to 1. It suggests that $\beta$ should not be too large. 
\end{itemize}

\vspace{0.3cm}

The rate of convergence of the mean squared error can be deduced from this theorem. We study the order of the remainder $d_n$ in order to exhibit the dominating terms. It appears that $d_n$ is the sum of three terms. The first one, with a exponential decay, is always neglectible as soon as $n$ is large enough, since $1> \zeta$. The two other are powers of $n$. Comparing their exponent, we can find the dominating term in function of $\gamma$ and $\beta$. Actually, there exists a rank $N_1(x,d)$ and some constants $C_5$ and $C_6(x, d)$ such that, for $n \geq N_0 +1$,

\begin{itemize}
\item[$\bullet$] if $\beta \leq 1-d \gamma$, then $d_n \leq C_5 n^{-2\gamma + \beta - 1}\;$,
\item[$\bullet$] if $\beta > 1-d \gamma$, then $d_n \leq C_6(x,d)n^{-\gamma + (1+\frac{1}{d})(\beta-1)}\;$.
\end{itemize}
Plugging these inequalities into Theorem \ref{ineg} leads to the following result. 
\begin{corollary}
\label{premier}
Under assumptions of Theorem \ref{ineg}, there exist ranks $N_4(x, \alpha, d)$ and constants $C_7(x, \alpha, d)$ and $C_8(x, \alpha)$ such that for all $n \geq N_4(x, \alpha, d)$,

\begin{itemize}
\item[$\bullet$] when $\beta > 1-d \gamma$ and $1- \beta < \zeta < \min \left( 1-\gamma, \left(1+ \frac{1}{d} \right) (1- \beta) \right)$, 
\[a_n(x) \leq \frac{C_7(d, x, \alpha, \zeta, \gamma)}{n^{-\zeta +\left( 1+ \frac{1}{d} \right) (1- \beta)}}\; ;\]
\item[$\bullet$] when $\beta \leq 1-d \gamma$, and 
$\zeta>\max(\beta-\gamma, \gamma-1)$,
\[a_n(x) \leq \frac{C_{8}(x, \alpha)}{n^{\gamma-\beta +1 - \zeta}}\; .\]
\end{itemize}
\end{corollary}

\begin{remark}
For other values of $\gamma$ and $\beta$, the derived inequalities do not imply the convergence to 0 of $a_n(x)$.  
\end{remark}
\vspace{0.3cm}

%

From this corollary, the \textit{optimal} choices for $(\beta, \gamma)$ can be derived, or more precisely parameters for which our upper-bound on the mean squared error decreases with the fastest rate.  

\begin{corollary}
\label{doubleoptim}
Under the same assumptions as in Theorem \ref{ineg}, the optimal choice is $\gamma= \frac{1}{1+d}$ with
$\zeta>\beta- \frac{1}{1+d}>0$ as small as possible.
With such parameters, there exists a constant $C_{9}(x, \alpha, d)$ such that $\forall n \geq N_4(x, \alpha, d)$,
\[a_n(x) \leq \frac{C_{9}(x, \alpha, d)}{n^{\frac{2}{1+d} + \frac{1-\beta-\zeta}{2}-\beta}}\;.\]
\end{corollary}
\vspace{0.3cm}

\textbf{Comments on the constant $\mathbf{C_9(x, \alpha, d)}$.}
Like all the other constants of this paper, we know the explicit expression of $C_9(x, \alpha, d)$. For a numerical example, see Subsection~\ref{subsec:ex1}. 

Notice that the constant $C_9(x, \alpha, d)$ depends on $x$ only through the lower bound $C_g(x)$ and the smoothness parameter $M(x)$. Often, $C_g(x)$ and $M(x)$ do not really depend on $x$ (see for example Subsection 3.1). In these cases (or when we can easily find a bound of $C_g(x)$ and $M(x)$ which do not depend on $x$), our result is uniform in $x$. Then, it is easy to deal with the integrated mean squared error and conclude that

\[\int_X a_n(x) f_X(x) dx \leq \frac{C_{9}(\alpha, d)}{n^{\frac{2}{1+d} + \frac{1-\beta-\zeta}{2}-\beta}}\;.\]
When $\alpha$ increases to 1, we try to estimate an extremal quantile. Then, $C_2(x, \alpha)$ becomes smaller and then $C_9(x, \alpha, d)$ increases: the bound deteriorates. This is because when $\alpha$ is large, the probability to sample on the right side of the quantile is small and the algorithm is less accurate.

Let us now comment on the dependency on the dimension $d$. The constant $C_9(x, d, \alpha)$ decreases when the dimension $d$ increases. Nevertheless, this tendency to decrease is too small to balance the behavior of the rate of convergence which is in $n^{\frac{-2}{1+d}}$, an illustration of the well-known curse of dimensionality.

\vspace{0.3cm}

\textbf{Comment on the rank $\mathbf{N_4(x, \alpha, d)}$.} This rank is the maximum of four ranks. There are two kinds of ranks. The ranks $(N_i)_{i \ne 0}$ depend on constants of the problem but are reasonably small, because the largest of them is the rank after which exponential terms are smaller than power of $n$ terms, or smaller power of $n$ terms are smaller than bigger power of $n$ terms. They often appear to be much smaller than $N_0$, which tends to be the limiting factor relevant for identifying optimal parameters (and at this stage the reasoning is no longer non-asymptotic). 

The rank $N_0$ is completely different. It was introduced in the first theorem because we could not deal with $a_n(x)$ directly. In fact it is the rank after which the deviation inequality, allowing us to use $b_n(x)$, is guaranteed to hold. It depends on the gap between $\zeta$ and $1- \beta$. The optimal $\zeta$ to obtain the rate of convergence of the previous corollary is $\zeta= 1 - \beta + \eta_{\zeta}$ with $\eta_{\zeta}$ as small as possible. The constant $\eta_{\zeta}$ appears on the rank $N_0$ and also on the rate of convergence (under the assumption that $N_4=N_0$ which is the case most of time)

\[\forall n \geq N_0=\exp \left( 2\eta_{\zeta}^{-1}\right), \, \, a_n(x)=\mathcal{O}\left(n^{\frac{-2}{1+d}+ \frac{\eta_{\zeta}}{2}+\beta}\right)\;.\] 
The smaller $\eta_{\zeta}$, the faster the rate of convergence, but also the larger the rank after which the inequalities hold. 

Let us give an example. For a budget of $N=1000$ calls to the code, one may choose $\eta_{\zeta}=0.3$ for the inequality to be theoretically true for $n=N$. Table \ref{tab1} gives the theoretical precision for different values of $d$ and compares it with the ideal case where $\eta_{\zeta}=0$. 

\begin{table}
\centering
\caption{Expected precision for the MSE when $N=1000$}
\label{tab1}
\begin{tabular}{|c|c|c|c|}
\hline 
$d$ & 1 & 2 & 3 \\ 
\hline 
$\eta_{\zeta}$=0.3 & 0.088 & 0.28 & 0.5 \\ 
\hline 
$\eta_{\zeta}$=0 & 0.031 & 0.1 & 0.17 \\ 
\hline 
\end{tabular} 
\end{table}

We can observe that, when $\eta_{\zeta}>0$, the precision increases with the dimension faster than when $\eta_{\zeta}=0$. Moreover, as soon as $\frac{1}{1+d}< \eta_{\zeta}/2$ ($d=6$ for our previous example), the result does not allow to conclude that $a_n$ decreases to 0 with this choice of $\eta_{\zeta}$.

Nonetheless, our simulation study (see next section) seem to indicate that this difficulty could be only an artifact of the proof: the introduction of $\zeta_n$ is required by the difficulty to compute  $\E\big[ (\theta_n(x)- \theta^*)P_n\big]$. In practice, the optimal rate of convergence for optimal parameters is reached early (see Section~3). 

\section{Numerical simulations}\label{sec:simu}

In this part we present some numerical simulations to illustrate our results. The following (simplistic) examples are chosen so as to be able to evaluate clearly the strengths and weaknesses of our algorithm:  the constants can be computed and the results can be interpreted easily. To begin with, we deal with dimension 1. We study two stochastic codes, differing by their smoothness.

\subsection{Dimension 1: square function}\label{subsec:ex1}

The first toy example is the very smooth code
\[g(X, \varepsilon)= X^2 + \varepsilon\]
where $X \sim \mathcal{U}([0, 1])$ and $\varepsilon \sim \mathcal{U}([-0.5, 0.5])$. We try to estimate the quantile of level $\alpha=0.95$ for $x=0.5$ and initialize our algorithm to $\theta_1=0.3$. We first check that our assumptions are fulfilled in this case. The conditional distribution of the output given $X=x$ is $ 
\mathcal{U}\big([-\frac{1}{2} + x^2 ; \frac{1}{2}+x^2]\big)$, and \[f_{(X, Y)}(u, v)= \mathbb{1}_{[- \frac{1}{2}+u^2, \frac{1}{2} + u^2] }(v)\,\mathbb{1}_{[0,1]}(u)\; .\]
Moreover, the code function $g$ takes its values in the compact set $[L_Y, U_Y]=[-\frac{1}{2}; \frac{3}{2}]$. Let us study assumption $\textbf{A1}$. If $a,b>0$ and if $B=[x-a, x+b]$ is an interval containing $x$, then
\begin{equation*}
\begin{aligned}
\left|F_{Y^B}(t)- F_{Y^x}(t) \right| & \leq  \left|\frac{ \int_{- \infty}^t \int_B f_{(X, Y)}(z, y) dy dz}{\int_B f_X(z) dz} - \int_{-\infty}^t f_{(X, Y)}(x, y) dy \right| \\
& \leq \frac{\int_{- \frac{1}{2}}^t \int_{x-a}^{x+b} \left| \mathbb{1}_{[- \frac{1}{2} + z^2; \frac{1}{2}+ z^2]} - \mathbb{1}_{[- \frac{1}{2} + z^2; \frac{1}{2}+ z^2]} \right|(y) dz dy}{\mu(B)} \;.\\
\end{aligned}
\end{equation*}
Now, we have to distinguish the cases in function of the localization of $t$. There are lots of cases, but computations are nearly the same. That is why we will develop only one case here. When $t \in [-\frac{1}{2}; x^2 - \frac{1}{2}]$, we have 

\begin{equation*}
\begin{aligned}
\left|F_{Y^B}(t)- F_{Y^x}(t) \right|& \leq   \frac{ \int_{x-a}^{x+b}  \int_{- \frac{1}{2}}^t  \left| \mathbb{1}_{[- \frac{1}{2} + z^2; \frac{1}{2}+ z^2]} - \mathbb{1}_{[- \frac{1}{2} + z^2; \frac{1}{2}+ z^2]} \right|(y)}{a+b} \\
& = \frac{ \int_{x-a}^{x+b} \left( \mathbb{1}_{z \geq x} (0) + \mathbb{1}_{z \leq x} ( t-z^2 + \frac{1}{2}) \mathbb{1}_{z \geq \sqrt{t+\frac{1}{2}}} \right) dz}{a+b}\\
& = \frac{\int_{x-a}^{x} (t+\frac{1}{2}-z^2) dz}{b+a}\;.\\
\end{aligned}
\end{equation*}
There are again two different cases. Since $t \in [-\frac{1}{2}; x^2 - \frac{1}{2}]$, we always have $(t + \frac{1}{2})^{\frac{1}{2}} \leq x$. But the position of $(t + 1/2)^{1/2}$ relative to $(x-a)$ is not always the same. If $t\in [- \frac{1}{2};-\frac{1}{2}(x-a)^2]$, we get

\begin{equation*}
\begin{aligned}
\left|F_{Y^B}(t)- F_{Y^x}(t)\right| & \leq \frac{\int_{x-a}^{x+b}(t-z^2+ \frac{1}{2}) dz}{b+a} \\
& \leq \left(t+ \frac{1}{2}\right)a - \frac{x^3}{3} + \frac{(x-a)^3}{3} \\
& \leq (x-a)^2 a - x^2a+a^2x- \frac{a^3}{3} \\
& \leq -a^2 x + \frac{2a^3}{3} \\
& \leq 0 + r_B \times 1^2 \times \frac{2}{3}\;, \\
\end{aligned}
\end{equation*}
as $0<a <1$. Finally, in this case, $\textbf{A1}$ is true with $M(x)= 2/3$. We can compute exactly in the same way for the other cases and we always find an $M(x)\leq 2/3$. The assumption $\textbf{A2}$ is also satisfied, taking $C_{input}=1$. We have already explained that assumption $\textbf{A3}$ is true for $[L_Y, U_Y]=[-1/2, 3/2]$. Finally assumption $\textbf{A4}$ is also satisfied with $C_{g}(x)=1$ and $C_2(x, \alpha)= 0.02$.

\subsubsection{Almost sure convergence}

Let us first deal with the almost sure convergence. We plot in Figure \ref{fig:conv5000}, for $(\beta, \gamma) \in [0,1]^2$, the relative error of the algorithm. Best parameters are clearly in the area $\beta > \gamma \geq 1/2$. We can even observe that for $\beta \approx 1$, $\beta \leq \gamma$ or $\gamma < 1/2$, the algorithm does not converge almost surely (or very slowly). This is in accordance with our theoretical results. Nevertheless, we can observe a kind of continuity for $\gamma$ around $1/2$ : in practice, the convergence becomes really slow only when $\gamma$ is significantly far away from $1/2$.

\begin{figure}[!ht]
$$\includegraphics[scale=0.5]{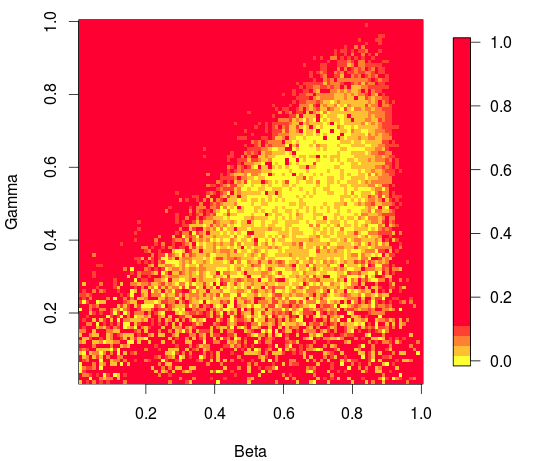}$$
\caption{Relative error for $n=5000$ dependence on $\beta$ and $\gamma$.}\label{fig:conv5000}
\end{figure}

\subsubsection{Mean Square Error (MSE)}
Let us study the best choice of $\beta$ et $\gamma$ in terms of $L^2$-convergence. We plot in Figure~\ref{fig:MSEcarre50} the mean squared error in function of $\gamma$ and $\beta$ (we estimate the MSE by a Monte Carlo method of 100 iterations). 

\begin{figure}[ht!]
\centering
\subfloat[Mean square error, $n=50$.]{\includegraphics[scale=0.3]{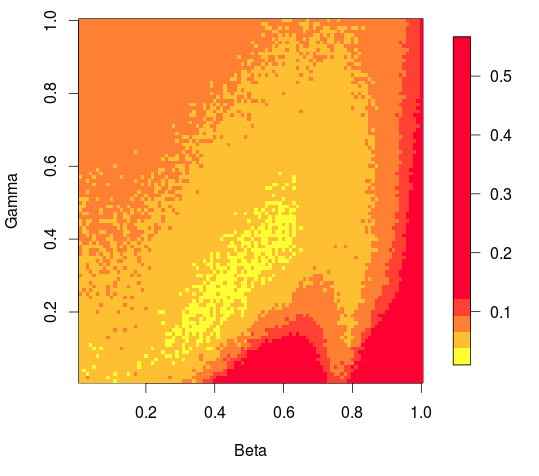}}
\subfloat[Mean square error, $n=200$.]{\includegraphics[scale=0.3]{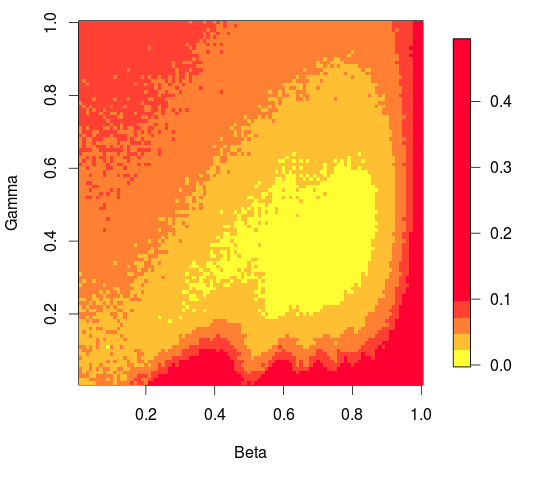}}
\caption{Mean squared error in function of $\beta$ and $\gamma$ for the square function.}\label{fig:MSEcarre50}
\end{figure}
Simulations confirm that the theoretical optimal area $\gamma=0.5$ and $\beta=\gamma+\eta_{\beta}$ gives the smallest MSE. Nevertheless, it seems that in practice we can relax the condition that \textit{the gap $\eta_{\beta}$ between $\beta$ and $\gamma$ is as small as possible}. Indeed, when $\eta_{\beta}$ is reasonably big, simulations show that we are still in the optimal area. 

In this case, we have at hand all the parameters to compute the theoretical bound of our theorems. In particular, in corollary \ref{doubleoptim}, we get

\[a_n(x)\leq \frac{C_{9}(x, d, \alpha)}{n^{\frac{1}{1+d}- \eta}}\;.\]
Table \ref{tab2} summarizes the value of the constants needed to compute the theoretical bound in this case. 

\begin{table}
\centering
\caption{Constant values}
\label{tab2}
\begin{tabular}{|c||c|c|c|c|c|c|c|c|c|c|c|c|}
\hline 
Constant& $\alpha$ & $M(x)$ & $C_{\text{input}}$ & $C_g(x)$ & $C_2(x, \alpha)$ & $U_Y-L_Y$  \\ 
\hline 
Value &0.95 & $\frac{2}{3}$ & 1 & 1 & 0.02 & 2 \\ 
\hline 
 Constant & $\sqrt{C_1}$ &$ C_3(d)$ & $C_4(d)$ & $C_5(x, d)$ & $C_6(x, d)$ & $C_9(x, d, \alpha)$ \\
\hline
 Value &2.95 & 7.39 & 2 & 1.95 & 12 & 180 \\
\hline
\end{tabular}

\end{table}
For $N=1000$, we obtain the bound $a_N(x) \leq 5.8$ which is over-pessimistic compared to the practical results. We can then think to a way to improve this bound. First of all, the constant $C_2(x, \alpha)$ is in fact not so small. Indeed, we have to take a margin in the proof, for the case where $\theta_n(x)$ goes out of $[L_Y, U_Y]$. This happens only with a very small probability. If we do not take this case into account, we have $C_2(x, \alpha)=1$. Then $C_{9}(x, \alpha, d) \approx 3.7$ and then, for $N=1000$, the bound is 0.11. Practical results are still better (we can observe that for $n=50$, we already have a MSE inferior to $0.05$), but the gap is less important.

\subsection{Dimension 1 - absolute value function}

Let us see what happens when the function $g$ is less smooth with respect to the first variable. We study the code 
\[g(X, \varepsilon) = |X| + \varepsilon\;,\]
where $X \sim \mathcal{U}\left( [-1, 1] \right)$ and $\varepsilon \sim \mathcal{U}\left( [-0.5, 0.5] \right)$.  We want to study the conditional quantile in $x=0$ (the point for which the differentiability fails). Assumptions can be checked as above. Since the almost surely convergence is true and gives really same kind of plots than the previous case, we only study the convergence of the MSE. In that purpose, we plot in Figure \ref{fig:MSEabs} the MSE  (estimated by 100 iterations of Monte Carlo simulations) in function of $\gamma$ and $\beta$, for n=300 (the discontinuity constraints us to make more iterations to have a sufficient precision) and $\theta_1=0.3$. Conclusions are the same than in the previous example concerning the best parameters. Nevertheless, we can observe that the lack of smoothness implies some remarkable behaviour around $\gamma=1$. 

\begin{figure}[!ht]
$$\includegraphics[scale=0.5]{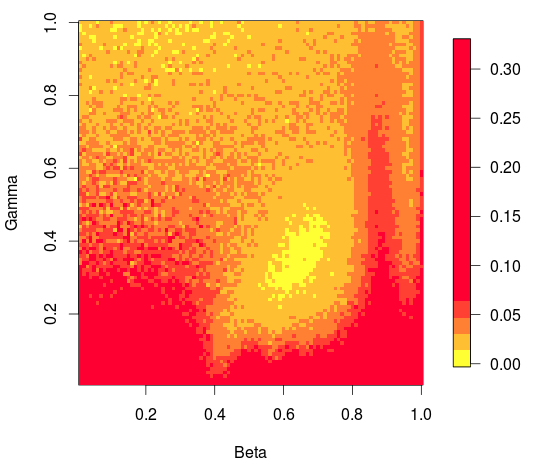}$$
\caption{MSE for $n=300$ in function of $\beta$ and $\gamma$ for absolute value function.}\label{fig:MSEabs}
\end{figure}

\subsection{Dimensions 2 and 3}

In dimension $d$, we showed that theoretical optimal parameters are $\gamma= \frac{1}{1+d}$ and $\beta= \gamma+ \eta$. To see what happens in practice, we still plot Monte Carlo estimations (200 iterations) of the MSE in function of $\gamma$ and $\beta$. 

\subsubsection{Dimension 2}

In dimension 2, we study two codes :
\[g_1(X, \varepsilon) = ||X||^2 + \varepsilon \, \, \text{and} \, \, g_2(X, \varepsilon)= X_1^2 + X_2 + \varepsilon\; , \]
where $X=(X_1, X_2) \sim \mathcal{U}\left( [-1, 1]^2 \right)$ and $\varepsilon \sim \mathcal{U}\left( [-0.5, 0.5] \right)$. In each case, we choose $n=400$ and want to study the quantile in the input point $x=(0, 0)$ and initialize our algorithm in $\theta_1=0.3$.  In Figure \ref{fig:d2}, we can see that $\beta=1$ and $\gamma=1$ are still really bad parameters. As in our theoretical results, $\gamma=\frac{1}{1+d}=\frac{1}{3}$ seems to be the best choice. Nevertheless, even if it is clear that $\beta< \gamma$ is a bad choice, the experiments seems to show that best parameter $\beta$ is strictly superior to $\gamma$, more superior than in theoretical case, where we take $\beta$ as close as possible of $\gamma$. As we said before, in practice, $N_0$ seems not to be the true limit rank. Indeed, with only $n=400$ iterations, in this case, the MSE, in the optimal parameters case reaches $0.06$.

\begin{figure}[ht!]
\centering
\subfloat[MSE, $n=400$, $d=2$, norm function $g_1$.]{\includegraphics[scale=0.3]{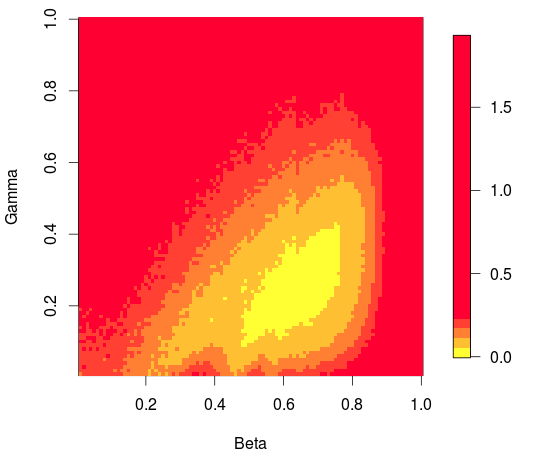}}
\subfloat[MSE, $n=400$, $d=2$, function $g_2$.]{\includegraphics[scale=0.3]{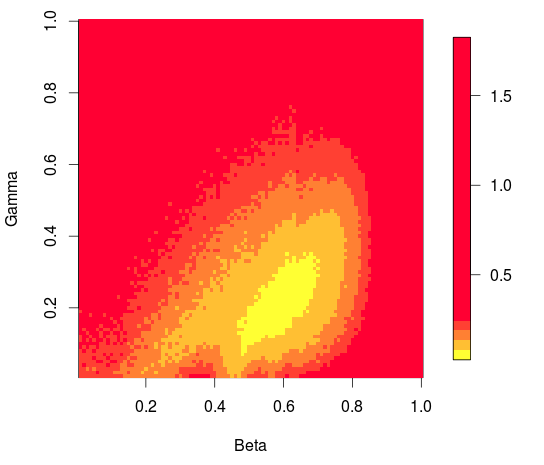}}
\caption{Mean square error in function of $\beta$ and $\gamma$.}\label{fig:d2}
\end{figure}

\subsubsection{Dimension 3}

In dimension 3, we study the two codes \[g_1(X, \varepsilon) = ||X||^2 + \varepsilon \quad \text{and} \quad g_2(X, \varepsilon)= X_1^2 + X_2 + \frac{X_3^3}{2} + \varepsilon\;,\]
where $X=(X_1, X_2, X_3) \sim \mathcal{U}\left( [-1, 1]^3 \right)$ and $\varepsilon \sim \mathcal{U}\left( [-0.5, 0.5] \right)$. In each case, we choose $n=500$ and want to study the quantile in the input point $(0, 0, 0)$. The interpretation of Figure \ref{fig:d3} are the same than in dimension 2. The scale is not the same, the convergence is slower again but with $n=500$ we nevertheless obtain a MSE of $0.10$.

\begin{figure}[ht!]
\centering
\subfloat[MSE, $n=500$, $d=3$, norm function $g_1$.]{\includegraphics[scale=0.3]{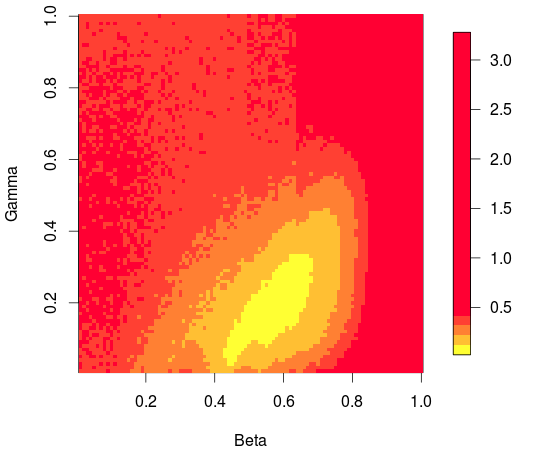}}
\subfloat[MSE, $n=500$, $d=3$, function $g_2$.]{\includegraphics[scale=0.3]{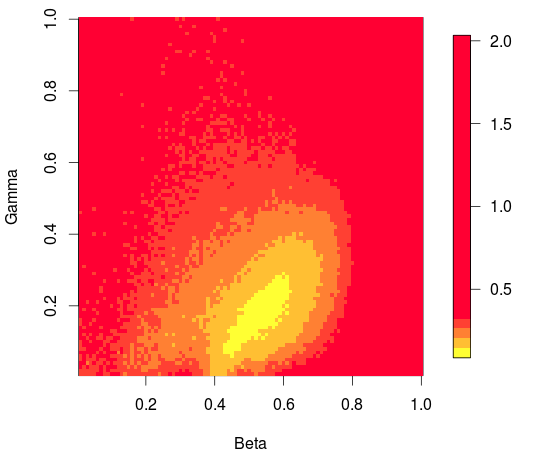}}
\caption{Mean squared error in function of $\beta$ and $\gamma$.}\label{fig:d3}
\end{figure}

\section{Conclusion and perspectives}

In this paper, we proposed a sequential method for the estimation of a conditional quantile of the output of a stochastic code where inputs lie in $\mathbb{R}^d$. We introduced a combination of $k$-nearest neighbors and Robins-Monro estimator. This algorithm has two parameters: the number of neighbors $k_n=\lfloor n^{\beta}\rfloor$ and the learning rate $\gamma_n=n^{-\gamma}$. By deriving a bias-variance decomposition of the risk, we showed that our algorithm is convergent for $\frac{1}{2} < \gamma < \beta < 1$ and we studied its mean squared error non-asymptotic rate of convergence. Moreover, we proved that the choice $\gamma=\frac{1}{1+d}$ and $\beta\gtrsim\gamma$ leads to the best rate of convergence. Numerical simulations show that the algorithm tuned with those theoretically optimal parameters is a powerful and accurate estimator of the conditional quantiles, even in dimension $d>1$. 

The theoretical guarantees are shown under strong technical assumptions, but our algorithm is a general methodology to solve the problem. Relaxing the conditions will be the object of a future work. Moreover, the proof that we propose constrained us to use an artefact parameter $\zeta$ which implies that the non-asymptotic inequality is theoretically true for large values of $n$, even if simulations confirm that this problem does not exist in practice. A second perspective is then to find a better way to prove this inequality for smaller $n$.
Finally, it would be of great interest to derive non-asymptotic lower-bounds for the mean squared error of the algorithm.

\appendix

\section{Technical lemmas and proofs}\label{sec:proofs}
\subsection{Technical lemmas and notation}

For sake of completeness, we start by recall some well-known facts on order statistics. 

\begin{lemma}
\label{Pn}
When $X$ has a density with respect to Lebesgue measure, denoting $P_n=\P( X \in kNN_{n+1}(x) | X_1, \dots X_n)$, we have the following properties

\begin{itemize}
\item[1)] $P_n=F_{||X-x||}\left( ||X-x||_{(k_{n+1},n)} \right)$\;,
\item[2)] $P_n \sim \mathrm{Beta}(k_{n+1}, n-k_{n+1}+1)$\;,
\item[3)] $ \E(P_n)= k_{n+1}/({n+1})$\;,
\item[4)] $\E(P_n^2)=\big(2k_{n+1} n - k_{n+1}^2 + 3k_{n+1} +k_{n+1} n^2 \big) / \big((n+1)^2(n+2)\big)$\;.
\end{itemize}
where we denote $F_{||X-x||}$ the cumulative distribution function of the random vector $||X-x||$, $||X-x||_{(k_{n+1},n)}$ the $k_{n+1}$ order statistic of the sample $(||X_1-x||, \dots, ||X_n-x||)$ and $\mathrm{Beta}(a,b)$ the beta distribution with parameters $a$ and $b$. 
\end{lemma}

\begin{proof}
Conditionally to $X_1, \dots, X_n$, the event \{$X \in kNN_{n+1}(x)$\} is equivalent to the event \{$||X-x|| \leq ||X-x||_{(k_{n+1}, n)}$\}. Then,

\begin{equation*}
\begin{aligned}
P_n & =\P(X \in kNN_{n+1}(x) | X_1 \dots X_n) \\
& = \P_{X} \left(||X-x|| \leq ||X-x||_{(k_{n+1},n)}|X_1 \dots X_n\right) \\
& = F_{||X-x||} \left(||X-x||_{(k_{n+1},n)}\right)\;. \\
\end{aligned}
\end{equation*}
Since $X$ has a density, the cumulative distribution function $F_{||X-x||}$ is continuous. Indeed, using the sequential characterization we get for a sequence $(t_n)$ converging to $t$

\begin{equation*}
\begin{aligned}
F_{||X-x||}(t_n)& = \P(X \in B_d(x, t_n) )\\
& = \int_{\mathbb{R}^d} f(z) \mathbb{1}_{B_d(x, t_n)}(z) \;\mathrm{d}z\;. \\
\end{aligned}
\end{equation*}
Since $f$ is integrable, the Lebesgue theorem allows us to conclude that

\begin{equation*}
\begin{aligned}
\lim_n \int_{\mathbb{R}^d} f(z) \mathbb{1}_{B_d(x, t_n)}(z) \;\mathrm{d}z& = \int_{\mathbb{R}^d} \lim_n f(z) \mathbb{1}_{B_d(x, t_n)}(z) \;\mathrm{d}z= \P(X \in B_d(x, t))\;,\\
\end{aligned}
\end{equation*}
so the cumulative distribution function is continuous. Then thanks to classical result on statistics order and quantile transform (see \cite{order}), we get 

\begin{equation*}
P_n  = F_{||X-x||} \left(||X-x||_{(k_{n+1}, n)}\right) \sim  U_{(k_{n+1},n)} \sim \beta(k_{n+1}, n-k_{n+1}+1) \;,
\end{equation*}
where we denoted $U_{(k_{n+1},n)}$ the $k_{n+1}$ statistic order of a independent sample of size $n$ distributed like a uniform law on $[0, 1]$.  
\end{proof}

Let us now recall some deviation results. 

\begin{lemma}
\label{bern}

We denote $\mathcal{B}(n, p)$ the binomial distribution of parameters $n$ and $p$, for $n \geq 1$ and $p \in [0, 1]$. Then, if $Z \sim \mathcal{B}(n, p)$, we get
\begin{equation*}
\begin{aligned}
\P  \left( \frac{Z}{n} < \frac{p}{2} \right) &  \leq \exp \left( - \frac{3np}{32} \right)\;,\\
\P \left( \frac{Z}{n} > 2p \right) & \leq \exp \left( - \frac{3np}{8} \right)\;.
\end{aligned}
\end{equation*}

\end{lemma}

\begin{proof}
Let $(Z_i)$ be an independent sample of Bernoulli of parameter $p$ and let \[Z =\frac{1}{n} \displaystyle\sum_{k=1}^n Z_i\;.\] We apply the Bernstein's inequality (see for example Theorem 8.2 in~\cite{Devroye}) to conclude that

\begin{equation*}
\begin{aligned}
\P(Z-p< - \zeta p )\leq \exp \left( - \frac{3np \zeta^2}{8}\right)\;,\\
\P(Z-p> \zeta p )\leq \exp \left( - \frac{3np \zeta^2}{8}\right)\;.\\
\end{aligned}
\end{equation*}
The results follow by taking $\zeta= \frac{1}{2}$ in the first case and $\zeta=1$ in the second case.  
\end{proof}

We now give some technical lemma useful to prove our main results. 

\begin{lemma}
\label{lemmeborne}
Suppose $\beta \geq \gamma$. Then, for every $C > 0$, we get

$$\P \left( \displaystyle\sum_{n} \gamma_n \mathbb{1}_{X_n \in kNN_n(x)} \leq C \right) = 0\;.$$
\end{lemma}

\begin{proof}
Let us denote $F$ the cumulative function distribution of $||X_n - x||$ and $ U_n = F(||X_n - x||)$, we get
\[\displaystyle\sum_{n} \gamma_n \mathbb{1}_{X_n \in kNN_n(x)}=\sum_{n} \gamma_n \mathbb{1}_{U_n \in kNN_n(0)}\;.\]
Hence, it is enough to show the desired result for $x=0$ and $X_j=U_j\sim \mathcal{U} \left( [0, 1] \right)$.

Let $\zeta$ be a positive real number. Let $N$ be an integer such that 
\begin{equation}
\label{zeta}
\begin{aligned}
\displaystyle\sum_{n \geq N} \exp \left( - \frac{3k_n}{16} \right) \leq \zeta\;.
\end{aligned}
\end{equation}
We set \[\Omega:= \left\lbrace \forall n \geq N, \, \, \displaystyle\sum_{j=1}^n \mathbb{1}_{U_j \leq \frac{k_n}{2n}} \leq k_n \right\rbrace \;.\]
On this event, for every $n \geq N$, there are at most $k_n$ elements $U_i$ such that $U_i$ is inferior to $\frac{k_n}{2n}$. Thus, if an element satisfies $U_j \leq \frac{k_n}{2n}$, it belongs to the $k_n$-nearest neighbors of $0$. Then, defining $Z_n:= \displaystyle\sum_{j=1}^n \mathbb{1}_{U_j \leq \frac{k_n}{2n}} \sim \mathcal{B}\left(n, \frac{k_n}{2n}\right)$,
\begin{equation}
\label{complem}
\begin{aligned}
\P \left( \overline{\Omega} \right) &\leq \displaystyle\sum_{n \geq N} \P \left( \displaystyle\sum_{j=1}^n \mathbb{1}_{U_j \leq \frac{k_n}{2n}} > k_n \right) \\
& \leq \displaystyle\sum_{n \geq N} \P \left( \frac{Z_n}{n} > \frac{k_n}{n} \right)\\
& \leq \displaystyle\sum_{n \geq N} \exp \left( - \frac{3k_n}{16} \right) \leq \zeta\; .\\
\end{aligned}
\end{equation}
by using the second inequality of Lemma \ref{bern} and Equation (\ref{zeta}). But, as we noticed above, on the event $\Omega$, we have
$
\mathbb{1}_{U_n \in kNN_n(0)} \geq \mathbb{1}_{U_n \leq \frac{k_n}{2n}}$; and thus
\begin{equation}
\label{inter}
\begin{aligned}
 \P \left( \Omega \cap \displaystyle\sum_{n \geq N} \gamma_n \mathbb{1}_{U_n \in kNN_n(0)} \leq C \right)\leq \P\left( \displaystyle\sum_{n \geq N} \gamma_n \mathbb{1}_{U_n \leq \frac{k_n}{2n}} \leq C \right) \;. \\
\end{aligned}
\end{equation}
Let now $(I_k)_k$ be a partition of $\llbracket N, + \infty \llbracket$ such that

\begin{equation*}
\begin{aligned}
\forall k \geq 1, \, \, \displaystyle\sum_{n \in I_k} \gamma_n \frac{k_n}{2n} \in [2C, 2C+1]\; .\\
\end{aligned}
\end{equation*}
Such a partition exists since, as $\beta \geq \gamma$, the sum $\displaystyle\sum_{n}\gamma_n \frac{k_n}{n}$ is divergent. Then, 
\begin{equation*}
\E \left[ \displaystyle\sum_{n\in I_k} \gamma_n \mathbb{1}_{U_n \leq \frac{k_n}{2n}} \right] = \displaystyle\sum_{n\in I_k}  \gamma_n \E\left(\mathbb{1}_{U_n \leq \frac{k_n}{2n}} \right)
 = \displaystyle\sum_{n\in I_k}  \gamma_n \frac{k_n}{2n}\; ,\\ \geq 2C \;.
\end{equation*}
and by independence, and since the variance of a Bernoulli variable is upper-bounded by its expectation,
\begin{equation*}
\Var\left[ \displaystyle\sum_{n \in I_k} \gamma_n \mathbb{1}_{U_n\leq \frac{k_n}{n}} \right] \leq \E \left[ \displaystyle\sum_{n \in I_k} \gamma_n \mathbb{1}_{U_n\leq \frac{k_n}{2n}} \right]  = \sum_{n \in I_k} \gamma_n \frac{k_n}{2n} \leq 2C+1\;.
\end{equation*}
Chebyshev's inequality yields:
\begin{equation*}
\begin{aligned}
\P \left( \displaystyle\sum_{n \in I_k} \gamma_n \mathbb{1}_{U_n\leq \frac{k_n}{2n}} \leq C \right) & \leq 
\P \left( \E\Big[\displaystyle\sum_{n \in I_k} \gamma_n \mathbb{1}_{U_n\leq \frac{k_n}{2n}}\Big] -\displaystyle\sum_{n \in I_k} \gamma_n \mathbb{1}_{U_n\leq \frac{k_n}{2n}} \geq 2C-C \right)\\ & \leq 
 \frac{2C+1}{C^2} \leq \frac{7}{9}
\end{aligned}
\end{equation*}
since $C \geq 3$. Thus, 
\begin{equation*}
\begin{aligned}
\P \left( \bigcap_{k} \left\lbrace \displaystyle\sum_{n \in I_k} \gamma_n \mathbb{1}_{U_n \leq \frac{k_n}{2n}} \leq C\right\rbrace\right) =0\;.\\
\end{aligned}
\end{equation*} 
and hence
\begin{equation}
\label{last}
\P \left( \displaystyle\sum_{n \geq N} \gamma_n \mathbb{1}_{U_n\leq \frac{k_n}{2n}} \leq C \right) =0\; .
\end{equation}
Thanks to (\ref{complem}), (\ref{inter}) and (\ref{last}), we get

\[\P \left( \displaystyle\sum_{n} \gamma_n \mathbb{1}_{U_n \in kNN_n(0)} \leq C \right)\leq \P \left( \displaystyle\sum_{n \geq N} \gamma_n \mathbb{1}_{U_n \in kNN_n(0)} \leq C \right) \leq \P(\overline{\Omega}) + 0 \leq \zeta \;,\]
which holds for all $\zeta >0$.
\end{proof}

\begin{lemma}
\label{dev}
Denoting $A_n$ the event $\{X_1, \dots, X_{n} \, \, | \, \,  P_{n} > \zeta_{n}\}$ where $\zeta_{n}= \frac{1}{n^{\zeta}}$ and the parameter $\zeta$ satisfies $1>\zeta > 1 - \beta$, we have for $n\geq 2^{1/(\zeta-(1-\beta))}$,

\[\P(A_n^C) \leq \exp \left( - \frac{3n^{1- \zeta}}{8} \right)\;.\]
\end{lemma}
\begin{proof}
Thanks to the Lemma \ref{Pn}, we obtain
\begin{equation*}
\begin{aligned}
\P(A_n^C) & = \P( \beta(k_{n+1}, n-k_{n+1}+1) \leq  \zeta_n) \\
& =  I_{\zeta_n}(k_{n+1}, n-k_{n+1}+1)\;, \\
\end{aligned}
\end{equation*}
where we denote $I_{\zeta}$ the incomplete $\beta$ function. A classical result (see \cite{Abra}) allows us to write this quantity in terms of the binomial distribution 
\begin{equation*}
\begin{aligned}
\P(A_n^C) & = \P( \mathcal{B}(n, \zeta_n) \geq k_{n+1}) \;.\\
\end{aligned}
\end{equation*}
Thanks to Lemma \ref{bern}, we know that

\[\P( \mathcal{B}(n, \zeta_n) \geq k_{n+1}) \leq \exp \left( - \frac{3n\zeta_{n}}{8} \right) \leq \exp \left(- \frac{3n^{1-\zeta}}{8} \right) \;,\]
as soon as $k_{n+1}/n \geq 2 \zeta_n$, which is true as soon as $n \geq 2^{1/(\zeta-(1-\beta))}$ because $\zeta > 1- \beta$.

%
%
%
%
\end{proof}
%
%
%
%

\begin{lemma}
\label{ps}
Under hypothesis of Theorem \ref{cvas}, $||X-x||_{(k_{n+1},n)}$ converges almost surely to 0. 
\end{lemma}

\begin{proof}
Let $u$ be a positive number. 
\begin{equation}
\begin{aligned}
\label{pu}
p_u : & = \P(X \in \mathcal{B}(x, u) )  = \int_{\mathcal{B}(x, u)} f(t) dt  \\
& \geq \mu_X \left( \mathcal{B}(x, u) \right)  = C_1 \frac{\pi^{\frac{d}{2}}}{\Gamma(\frac{d}{2}+1)} \\
& = C_{input} C_4(d)u^d =: q_u\;.\\
\end{aligned}
\end{equation}
Let $Z$ be a random variable of law $\mathcal{B}(n, p_u)$. Since $||X-x||_{(k_{n+1},n)} >u$ implies that there are at the most $k_{n+1}$ elements of the sample which satisfy $X \in \mathcal{B}(x, q_u)$, we get :
\begin{equation*}
\begin{aligned}
\P( ||X-x||_{(k_{n+1},n)} >u) = \P( Z < k_{n+1})\;. \\
\end{aligned}
\end{equation*}
Thanks to equation $\eqref{pu}$, and denoting $\tilde{Z}$ a random variable of law $\mathcal{B}(n, q_u)$, we have

\begin{equation*}
\begin{aligned}
\P( ||X-x||_{(k_{n+1})} >u) \leq \P( \tilde{Z} < k_{n+1})\;. \\
\end{aligned}
\end{equation*}
Lemma \ref{bern} implies that $\P( ||X-x||_{(k_{n+1})} >u)$ is the general term of a convergent sum. Indeed, when $n$ is large enough, then $k_{n+1}/n < q_u/2$ because $k_{n+1}/n$ converges to 0 ($\beta<1$). The Borel-Cantelli Lemma then implies that $||X-x||_{(k_{n+1},n)}$ converges almost surely to 0. 
\end{proof}

\begin{lemma}
\label{unif}
With the same notation as above, 

\[\E(P_n ||X-x||_{(k_{n+1},n)}  ) \leq C_3(d) \left( \frac{k_{n+1}}{n+1} \right)^{1+\frac{1}{d}}\;.\]
\end{lemma}

\begin{proof}

Let us denote $\tilde{F}$ and $\tilde{f}$ the cumulative and density distribution function of the law of $||X-x||$. 
\begin{equation*}
\begin{aligned}
\E(||X-x||_{(k_{n+1}, n)} P_n) & = \E \left( ||X-x||_{(k_{n+1}, n)} \tilde{F} \left( ||X-x||_{(k_{n+1}, n)} \right) \right) \\
&= \int y \tilde{F}(y) f_{||X-x||_{(k_{n+1}, n)}}(y) dy\;, \\
\end{aligned}
\end{equation*}
with

\begin{equation*}
\begin{aligned}
f_{|X-x|_{(k_{n+1}, n)}}(y)& = \frac{n!}{(k_{n+1}-1)!(n-k_{n+1})!} \tilde{F}(y)^{k_{n+1}-1}\left(1- \tilde{F}(y) \right)^{n-k_{n+1}} \tilde{f}(y)\;. \\
\end{aligned}
\end{equation*}
Then we get

\begin{equation*}
\begin{aligned}
\E(||X-x||_{(k_{n+1}, n)} P_n) & = \int y \tilde{F}(y)^{k_{n+1}} (1 - \tilde{F}(y))^{n-k_{n+1}} \tilde{f}(y) \frac{n!}{(k_{n+1}-1)!(n-k_{n+1})!} \\
&= \frac{k_{n+1}}{n+1} \E \left(||X-x||_{(k_{n+1}+1, n+1)} \right)\;. \\
\end{aligned}
\end{equation*}
We denote $U_{|.|}$ the upper bound of the support of $||X-x||$, and write

\begin{equation*}
\begin{aligned}
\E(||X-x||_{(k_{n+1}+1, n+1)} ) = \int_{0}^{U_{|.|}} \P( ||X-x||_{(k_{n+1}+1, n+1)} >u) du\;.
\end{aligned}
\end{equation*}
Using same arguments that in Lemma \ref{cvas}, denoting $C_{10}(d)=\sqrt[d]{\frac{2(k_{n+1}+1)}{(n+1)C_{input}C_4(d)}}$,we get
 
 \begin{small}
 \begin{equation*}
 \begin{aligned}
 I:=\int_{0}^{U_{|.|}} \P( ||X-x||_{(k_{n+1}+1, n+1)} >u) du  = & \int_0^{C_{10}(d)} \P (\mathcal{B}(n+1 , q_u) < k_{n+1}+1) du \\
 & + \int_{C_{10}(d)}^{U_{|.|}} \P (\mathcal{B}(n+1, q_u) < k_{n+1}+1) du \\
  \leq & \int_0^{C_{10}(d)} 1 du \\
 & +  \int_{C_{10}(d)}^{U_{|.|}} \exp \left( - \frac{3(n+1) C_{input}C_4(d) u^d }{32}  \right)du \;,
 \end{aligned}
 \end{equation*}
 \end{small}
where we use Lemma \ref{bern} in the second integral because $u > C_{10}(d)$ implies $\frac{k_{n+1}+1}{n+1}< \frac{q_u}{2}$. Then, we obtain

 \begin{equation*}
 \begin{small}
 \begin{aligned}
 I  & \leq  C_{10}(d) + \int_{C_{11}(d)}^{+ \infty} \exp \left( - \frac{3(n+1) C_{input}C_4(d) u^d }{32}  \right)du \\
  & \leq C_{10}(d) + \int_0^{+ \infty} \frac{u^{d-1}}{{C_{10}(d)}^{d-1}} \exp \left( - \frac{3(n+1) C_{input}C_4(d)u^d }{32}  \right) du \\
  & = C_{10}(d) + \frac{C_{11}(d)}{{C_{10}(d)}^d} \frac{32}{3(n+1) d C_{input} C_4(d)} \left[ - \exp \left( -\frac{3(n+1) C_{input}C_4(d) u^d }{32}  \right) \right]^{+ \infty}_0 \\
  & = C_{10}(d) \left( 1 + \frac{3(n+1) d C_{input}C_4(d)}{32 C_{10}(d)^d }\right) \\
  & = \sqrt[d]{\frac{2(k_{n+1}+1)}{(n+1) C_{input}C_4(d)}} \left( 1+ \frac{16}{3 d (k_{n+1}+1)} \right) \\
  &  = \sqrt[d]{\frac{k_{n+1}}{n+1}} \left[ \sqrt[d]{\frac{2}{C_{input}C_4(d)}} \sqrt[d]{\frac{k_{n+1}+1}{k_{n+1}}} \left(1+ \frac{16}{3d(k_{n+1}+1)} \right) \right] \\
  & \leq \sqrt[d]{\frac{k_{n+1}}{n}} \sqrt[d]{\frac{4}{C_{input}C_4(d)}} \left( 1 + \frac{8}{3d} \right) \\
  & =: C_3(d) \sqrt[d]{\frac{k_{n+1}}{n+1}} \;,\\
\end{aligned}
\end{small}
\end{equation*}
because for $n \geq 1$, we get $k_n\geq 1$. 
\end{proof}

\begin{lemma}
\label{Guillaume}
Let $(b_n)$ be a a real sequence. If there exist sequences $(c_n)_{n\geq 1} \in [0, 1]^{\mathbb{N}}$ and $(d_n)_{n\geq 1}\in ]0, + \infty[^{\mathbb{N}}$ such that 

\[\forall n \geq N_0, \, \, b_{n+1} \leq b_n (1 -  c_{n+1}) + d_{n+1}\;,\]
then for all $n \geq N_0+1$, 

\[ \forall n, \, \, b_n\leq \exp \left( - \displaystyle\sum_{k=1}^n{N_0+1} c_k \right)b_{N_0} + \displaystyle\sum_{k=N_0+1}^{n} \exp \left( - \left( \displaystyle\sum_{j=1}^n c_j  - \displaystyle\sum_{j=1}^k c_j\right) \right)d_k\;.\]
\end{lemma}

\begin{proof}
This inequality appears in \cite{moul} and references therein. It can be proved by induction using that $\forall x \in ]0, + \infty[, \, \, \exp(x) \geq 1+x$.
\end{proof}

Let us first prove the following consequence of Assumption \textbf{A3}.

\begin{lemma}
\label{borne}
Under assumption \textbf{A3}, if $\beta \geq \gamma$, then for all $x$ and for all $n \geq 1$, $$\theta_n(x)\in [L_Y-(1-\alpha), U_Y + \alpha], a.s.$$
\end{lemma}
\vspace{0.3cm}

\begin{proof}

Suppose that $\theta_n(x)$ leaves the compact set $[L_Y, U_Y]$ by the right at step $N_0$. By definition, $\theta_{N_0-1}\leq U_Y$ and consequently $\theta_{N_0} \leq U_Y + \alpha \gamma_{N_0}$. At next step, since $\theta_{N_0}> U_Y$, we have $Y_{N_0+1}\leq \theta_{N_0}$ and then 

\[\theta_{N_0+1}\leq U_Y + \alpha \gamma_{N_0}-(1-\alpha) \gamma_{N_0+1}\mathbb{1}_{X_{N_0+1}\in kNN_{N_0+1}(x)}\;.\]
Then, the algorithm either does not move (if $X_{N_0+1} \notin kNN_{N_0+1}(x)$) or comes back in direction of $[L_Y, U_Y]$ with a step of $(1-\alpha) \gamma_{N_0+1}$. Then, if 

\[\displaystyle\sum_{n \geq0} \gamma_{n}\mathbb{1}_{X_{n} \in kNN_{n}(x)}=+ \infty \, \, a.s\;,\]
the algorithm almost surely comes back to the compact set $[L_Y, U_Y]$. Thanks to Lemma \ref{lemmeborne}, we know that, since $\beta \geq \gamma$, the previous sum diverges almost surely.
A similar result holds when the algorithm leaves the compact set by the left and finally we have shown that almost surely as $\gamma_n\leq 1$,
\[\theta_n(x) \in [L_Y-(1-\alpha), U_Y + \alpha]=:[L_{\theta_n}, U_{\theta_n}]\;.\]
\end{proof}

\subsection{Proof of Theorem \ref{cvas} : almost sure convergence}

To prove this theorem, we adapt the classical analysis of the Robbins-Monro algorithm (see \cite{blum}). In the sequel we do not write $\theta_n(x)$ but $\theta_n$ to make the notation less cluttered. 

\subsubsection{Martingale decomposition}

In this sequel, we still denote $H(\theta_n, X_{n+1}, Y_{n+1}):= \left( \mathbb{1}_{Y_{n+1} \leq \theta_n} - \alpha \right) \mathbb{1}_{X_{n+1} \in kNN_{n+1}(x)}$, $\mathcal{F}_n=\sigma (X_1, \dots, X_n, Y_1, \dots, Y_n)$ and $\P_n$ and $\E_n$ the probability and expectation conditionally to $\mathcal{F}_n$. We introduce

\begin{equation*} 
\begin{aligned}
h(\theta_n)  :&= \E( H( \theta_n, X_{n+1}, Y_{n+1}) | \mathcal{F}_n) \\
& = \P_n (Y_{n+1} \leq \theta_n \cap X_{n+1} \in kNN_{n+1}(x) )-  \alpha \P_n(X_{n+1} \in kNN_{n+1}(x) ) \\
& = P_{n} \left[\left( F_{Y^{kNN_{n+1}(x)}}(\theta_n\right) - F_{Y^x}(\theta^*(x) ) \right]\; . \\
\end{aligned}
\end{equation*}
Then,
\begin{equation*}
\begin{aligned}
T_n  = \theta_n + \displaystyle\sum_{j=1}^n \gamma_j h_{j-1}( \theta_{j-1}) 
 = \theta_0(x) - \displaystyle\sum_{j=1}^n \gamma_j \xi_j\; , \\
\end{aligned}
\end{equation*}
with $\xi_j = H(\theta_{j-1}, X_j, Y_j)-h_{j-1}(\theta_{j-1})$ is a martingale. It is bounded in $\mathbb{L}^2(\mathbb{R})$. Since 
\[\displaystyle{\sup_{n}}|\xi_n|  \leq \alpha + (1+\alpha) = 1+2 \alpha, \;\]
the Burkholder inequality gives the existence of a constant $C$ such that
\begin{equation*}
\begin{aligned}
\E( |T_n|^2)  \leq \E \left( \left( \displaystyle\sum_{j=1}^n \gamma_j \xi_j \right)^2 \right)  \leq C \E \left( \left| \displaystyle\sum_{j=1}^n \left( \gamma_j \xi_j \right)^2 \right|^2 \right)  \leq C(1+2 \alpha) \displaystyle\sum_{j=1}^n \gamma_j^2< \infty \; .
\end{aligned}
\end{equation*}

\subsubsection{The sequence $(\theta_n)$ converges almost surely}

First, let us prove that 

\begin{equation}
\label{dv}
\P(\theta_n \to \infty) + \P(\theta_n \to - \infty )= 0.
\end{equation}
Let us suppose that this probability is positive (we name $\Omega_1$ the non-negligeable set where $\theta_n(\omega)$ diverges to $+ \infty$ and the same arguments would show the result when the limit is $- \infty$). Let $\omega$ be in $\Omega_1$. We have $\theta_n(\omega) \leq \theta^*$ for only a finite number of $n$. 

Let us show that on an event $\Omega \subset \Omega_1$ with positive measure, for $n$ large enough, $h(\theta_n(\omega)) >0$. First, we know that $P_{n}$ follows a Beta distribution. This is why $\forall n, \, \, \P( P_{n}=0)=0$. Then, the Borel-Cantelli Lemma gives that $$\P( \exists N \, \,  \forall n \geq N \, \, P_{n}>0)= 1\; .$$ As $\Omega_1$ has a positive measure, we know that there exists $\Omega_2\subset \Omega_1$ with positive measure such that $\forall \omega \in \Omega_2$, $\theta_n(\omega) \rightarrow + \infty$ and for all $n$ large enough, $P_{n}(\omega) > 0$. Since $$h(\theta_n(\omega))=P_{n} \left( F_{Y^{B_n^{k_{n+1}}(x)}}(\theta_n(\omega)) - \alpha \right)\;,$$ we have now to show that on $\Omega \subset \Omega_2$ of positive measure, $$F_{Y^{B_n^{k_{n+1}}(x)}}(\theta_n(\omega)) - \alpha > 0\; .$$ As $\theta_n(\omega)$ diverges to $+ \infty$, we can find $D$ such that for $n$ large enough, $\theta_n(\omega) > D > \theta^*$. Then,
\begin{equation*}
\begin{aligned}
F_{Y^{B_n^{k_{n+1}}(x)}}(\theta_n(\omega)) - \alpha & = F_{Y^{B_n^{k_{n+1}}(x)}}(\theta_n(\omega)) - F_{Y^x}(\theta^*(x)) \\
& =  F_{Y^{B_n^{k_{n+1}}(x)}}(\theta_n(\omega)) - F_{Y^{B_n^{k_{n+1}}(x)}}(D)
\\ & + F_{Y^{B_n^{k_{n+1}}(x)}}(D)  - F_{Y^x}(D) + F_{Y^x}(D) - F_{Y^x}(\theta^*(x))\;. \\
\end{aligned}
\end{equation*}
First, $F_{Y^{B_n^{k_{n+1}}(x)}}(\theta_n(\omega)) - F_{Y^{B_n^{k_{n+1}}(x)}}(D) \geq 0$ because a cumulative distribution function is non-decreasing. Then, we set $\eta= F_{Y^x}(D) - F_{Y^x}(\theta^*(x))$ which is a finite value. To deal with the last term, we use our assumption $\textbf{A1}$. 

\begin{equation*}
\begin{aligned}
F_{Y^{B_n^{k_{n+1}}(x)}}(D)  - F_{Y^x}(D) & \geq -M(x) ||X-x||_{(k_{n+1}, n)}\;.  \\
\end{aligned}
\end{equation*}
We know, thanks to Lemma \ref{ps}, that $||X-x||_{(k_{n+1}, n)}$ converges almost surely to 0. Then, there exists a set $\Omega_3 \subset \Omega_1$ of probability strictly non-negative such that forall $\omega$ in  $\Omega_3$, the previous reasoning is true. And for $\zeta < \frac{\eta}{L}$, there exists rank $N(\omega)$ such that if $n \geq N$, 

\begin{equation}
\begin{aligned}
\label{tut}
F_{Y^{B_n^{k_{n+1}}(x)}}(D)  - F_{Y^x}(D) & \geq 0  -L \zeta + \eta >0 \; .\\
\end{aligned}
\end{equation}
Finally, for $\omega \in \Omega_3$ (set of strictly non-negative measure), we have shown that after a certain rank, $h(\theta_n(\omega))>0$. This implies that on $\Omega_3$ of positive measure,

\[\underset{n}{\lim} \left[ \theta_n(\omega) + \displaystyle\sum_{j=1}^{n}\gamma_{j-1} h_{j-1}(\theta_{j-1}(\omega)) \right] = + \infty\;,\]
which is absurd because in the previous part we proved that $T_n$ is almost surely convergent. Then $\theta_n$ does not diverge to $+ \infty$ or $- \infty$. 

\vspace{0.3cm}

Now, we will show that $(\theta_n)$ converges almost surely. In all the sequel of the proof, we reason $\omega$ by $\omega$ like in the previous part. To make the reading more easy, we do not write $\omega$ and $\Omega$ any more. Thanks to Equation (\ref{dv}) and to the previous subsection, we know that, with probability positive, there exists a sequence $(\theta_n)$ such that

\begin{equation*}
\left\{
\begin{aligned}
(a) \, \, & \theta_{n} + \displaystyle\sum_{j=1}^n \gamma_{j-1} h(\theta_{j-1}) \, \, \text{converges to a finite limit}\\
(b) \, \, & \liminf \theta_n < \limsup \theta_n \;.
\end{aligned}
\right.
\end{equation*}
Let us suppose that $\limsup \theta_n > \theta^*$ (we will find a contradiction, the same argument would allow us to conclude in the other case). Let us choose $c$ and $d$ satisfying $c > \theta^*$ and $\liminf \theta_n < c < d < \limsup \theta_n.$ Since the sequence $(\gamma_n)$ converges to 0, and since $(T_n)$ is a Cauchy sequence, we can find a deterministic rank $N$ and two integers $n$ and $m$ such that $N \leq n < m$ implies

\begin{equation*}
\left\{
\begin{aligned}
(a) \, \, & \gamma_{n} \leq \frac{(d-c)}{3(1- \alpha)} \\
(b) \, \, & \left|\theta_m - \theta_n - \displaystyle\sum_{j=n}^{m-1}\gamma_{j}h(\theta_{j-1}) \right|\leq  \frac{d-c}{3}\;.
\end{aligned}
\right.
\end{equation*}
We choose $m$ and $n$ so that 
\begin{equation}
\left\{
\begin{aligned}
(a) \, \, & N \leq n <m \\
(b) \, \, & \theta_n< c, \, \, \theta_m > d\\
(c) \, \, & n<j<m \Rightarrow  c \leq \theta_j \leq d\;.
\end{aligned}
\right.
\label{ecart}
\end{equation}
This is possible since beyond $N$, the distance between two iterations will be either
\[\alpha \gamma_{n} \leq  \frac{\alpha(d-c)}{3(1- \alpha)} < (d-c)\;,\] because $\alpha<\frac{3}{5}$ or \[(1- \alpha) \gamma_{n} \leq \frac{1}{3}(d-c) < (d-c)\;.\]
Moreover, since $c$ and $d$ are chosen to have an iteration inferior to $c$ and an iteration superior to $b$, the algorithm will necessarily go through the segment $[c, d]$. We then take $n$ and $m$ the times of enter and exit of the segment. Now, 

\begin{equation*}
\begin{aligned}
\theta_m-\theta_n & \leq  \frac{d-c}{3} + \displaystyle\sum_{j=n}^{m-1}\gamma_{j+1}h_j(\theta_{j})\\
 & \leq \frac{d-c}{3}  + \gamma_{n+1}h(\theta_{n})\;,\\
\end{aligned}
\end{equation*}
because $n < j < m$, we get $\theta^* < c < \theta_j$ and we have already shown that in this case, $h_j(\theta_j) >0$. We then only have to deal with $\theta_n$. If $\theta_n > \theta^*$, we can apply the same result and then 

\[\theta_m - \theta_n \leq \frac{d-c}{3}\;,\]
which is in contradiction with (b) of equation \eqref{ecart}. When $\theta< \theta^*$, 

\begin{equation*}
\begin{aligned}
\theta_m-\theta_n & \leq  \frac{d-c}{3} + \gamma_{n}h_{n-1}(\theta_{n-1}) \\
 & \leq \frac{d-c}{3} + \gamma_{n}(1- \alpha) \\
 & \leq \frac{d-c}{3} + \frac{d-c}{3} < (d-c) \;,
\end{aligned}
\end{equation*}
which is still a contradiction with (b) of \eqref{ecart}. We have shown that the algorithm converges almost surely. 

\subsubsection{The algorithm converges almost surely to $\theta^*$}
Again we reason by contradiction. Let us name $\theta$ the limit such that $\P(\theta \ne \theta^*) > 0$. With positive probability, we can find a sequence $(\theta_n)$ which converges to $\theta$ such that 

\begin{equation*}
\left\{
\begin{aligned}
(a) \, \,& \theta^* < \zeta_1 < \zeta_2 < \infty \\
(b) \, \,& \zeta_1 < \theta < \zeta_2\;,
\end{aligned}
\right.
\end{equation*}
(or $- \infty < \zeta_1 < \zeta_2 < \theta^*$ but arguments are the same in this case). Then, for $n$ large enough, we get

\begin{equation*}
\begin{aligned}
\zeta_1 < \theta_n < \zeta_2 \;. 
\end{aligned}
\label{encadrement}
\end{equation*}
Finally, on the one hand, $(T_n)$ and $(\theta_n)$ are convergent, and we also know that the sum $ \sum \gamma_{j+1} h_j(\theta_j)$ converges almost surely. Let us then show that on the other hand, $h(\theta_n)= P_{n} ( F_{Y^{B_n^{k_{n+1}}(x)}}(\theta_n) - \alpha)$ is lower bounded. First we know thanks to Lemma \ref{dev}, that for $1> \zeta > 1- \beta$ and $\zeta_n=\frac{1}{n^{\zeta}}$,

\[\P(P_{n} \leq \zeta_{n}) \leq\exp \left( - \frac{3n^{1- \zeta}}{8} \right)\;.\]
This is the general term of a convergent sum. Therefore, the Borel-Cantelli Lemma gives
\[\P( \exists N \, \, \forall n \geq N \, \, P_{n} > \zeta_{n} )=1\;.\]
Moreover, as we have already seen in Equation \eqref{tut}, since $\theta_n > \zeta_1 > \theta^*$,

\begin{equation*}
\begin{aligned}
F_{Y^{B_n^{k_{n+1}}(x)}}(\theta_n)  - \alpha & \geq 0 - M(x)||X-x||_{(k_{n+1}, n)} + F_{Y^x}(\zeta_1) - F_{Y^x}(\theta^*(x))\;. \\
\end{aligned}
\end{equation*}
Then, when $n$ is large enough so that \[||X-x||_{(k_{n+1}, n)} \leq \frac{F_{Y^x}(\zeta_1) - F_{Y^x}(\theta^*(x))}{2M(x)}\]
holds, we have
\[F_{Y^{B_n^{k_{n+1}}(x)}}(\theta_n)  - \alpha \geq \frac{F_{Y^x}(\zeta_1) - F_{Y^x}(\theta^*(x))}{2}\;.\]
Finally there exists a set $\Omega$ of positive probability such that, $\forall \omega \in \Omega$

\begin{equation*}
\begin{aligned}
\displaystyle\sum_{k=1}^n\gamma_{k+1}h_k(\theta_{k}) & \geq \frac{F_{Y^x}(\zeta_1) - F_{Y^x}(\theta^*(x))}{2} \displaystyle\sum_{k=1}^n \gamma_{k+1} P_{k}\geq \displaystyle\sum_{k=1}^n \frac{1}{(k+1)^{\gamma + \zeta}}\;,\\
\end{aligned}
\end{equation*}
which is a contradiction (with the one hand point) because the sum is divergent ($\gamma + \zeta <1$).
\subsection{Proof of Theorem \ref{ineg} : Non-asymptotic inequality on the mean squared error.}
Let $x$ be fixed in $[0, 1]$. We want to find an upper-bound for the mean squared error $a_n(x)$ using Lemma \ref{Guillaume}. In the sequel, we will need to study $\theta_n(x)$ on the event $A_n$ of the Lemma \ref{dev}. Then, we begin to find a link between $a_n(x)$ and the mean squared error on this event. 

\begin{equation}
\label{debutant}
\begin{aligned}
a_n(x) & = \E \left[ \left( \theta_n(x) - \theta^*(x) \right)^2 \mathbb{1}_{A_n} \right] + \E \left[ \left( \theta_n(x) - \theta^*(x) \right)^2 \mathbb{1}_{A_n^C} \right] \\
& \leq \E \left[ \left( \theta_n - \theta^*(x) \right)^2 \mathbb{1}_{A_n} \right] + C_1 \P(A_n^C)\\
& \leq \E \left[ \left( \theta_n(x) - \theta^*(x) \right)^2 \mathbb{1}_{A_n} \right] + C_1\exp \left( - \frac{3n^{1- \zeta}}{8} \right)\; ,
\end{aligned}
\end{equation}
thanks to Lemma \ref{dev} and for $n \geq N_0$.

Let us now study the sequence $b_n(x):=\E \left[ \left( \theta_n(x) - \theta^*(x) \right)^2 \mathbb{1}_{A_n} \right]$. First, for $n \geq 0$,

\begin{equation*}
\begin{aligned}
b_{n+1}(x) & \leq \E \left[ (\theta_{n+1}(x) - \theta^*(x) )^2 \right]\;.\\
\end{aligned}
\end{equation*}
But,

\begin{small}
\begin{equation*}
\begin{aligned}
(\theta_{n+1}(x)- \theta^*(x))^2 & = (\theta_n(x)- \theta^*(x))^2 \\ & + \gamma_{n+1}^2 \left[ (1 - 2 \alpha) \mathbb{1}_{Y_{n+1} \leq \theta_n(x)} + \alpha^2 \right] \mathbb{1}_{X_{n+1} \in kNN_{n+1}(x)} \\
& - 2 \gamma_{n+1}(\theta_n(x) - \theta^*(x))\left( \mathbb{1}_{Y_{n+1} \leq \theta_n(x)} - \alpha \right) \mathbb{1}_{X_{n+1} \in kNN_{n+1}(x)}\;. \\
\end{aligned}
\end{equation*}
\end{small}
Taking the expectation conditional to  $\mathcal{F}_n$, as $[(1 - 2 \alpha) \mathbb{1}_{Y_{n+1} \leq \theta_n(x)} + \alpha^2] \leq 1$, we get

\begin{small}
\begin{equation*}
\begin{aligned}
 \E_n \left( (\theta_{n+1}(x)- \theta^*(x))^2 \right)  & \leq \E_n \left( \left(\theta_{n}(x)- \theta^*(x) \right)^2 \right)
 + \gamma_{n+1}^2 \P_n \left(X_{n+1} \in kNN_{n+1}(x) \right) \\
& -  2 \gamma_{n+1}\left(\theta_n(x)-\theta^*(x) \right) \left[ \P_n \left(Y_{n+1} \leq \theta_n(x) \cap X_{n+1} \in kNN_{n+1}(x) \right)\right. \\
& - \left. \alpha \P_n \left(X_{n+1} \in kNN_{n+1}(x) \right) \right]\;. \\
\end{aligned}
\end{equation*}
\end{small}
Using the Bayes formula, we get

\begin{small}
\begin{equation*}
\begin{aligned}
\E_n \left( \theta_{n+1}(x)- \theta^*(x))^2 \right) & \leq \E_n \left( \left(\theta_{n}(x)- \theta^*(x) \right)^2 \right) + \gamma_{n+1}^2 P_{n}\\
& -  2 \gamma_{n+1} \left(\theta_n(x)-\theta^*(x) \right) P_{n} \left[ F_{Y^{B_n^{k_{n+1}}(x)}}(\theta_n(x)) - F_{Y^x}(\theta^*(x))  \right]\;, \\
\end{aligned}
\end{equation*}
\end{small}
Let us split the double product into two terms representing the two errors we made by iterating our algorithm. 
\begin{small}
\begin{equation}
\begin{aligned}
\label{hypo}
\E_n \left( \theta_{n+1}(x)- \theta^*(x))^2 \right) & \leq  \left(\theta_{n}(x)- \theta^*(x) \right)^2  + \gamma_{n+1}^2 P_{n+1} \\
& -  2 \gamma_{n+1}\left(\theta_n(x)-\theta^*(x) \right)  P_{n+1} \left[ F_{Y^{B_n^{k_{n+1}}(x)}}(\theta_n(x)) - F_{Y^x}(\theta_n(x)) \right] \\
& -  2 \gamma_{n+1} \left(\theta_n(x)-\theta^*(x) \right)  P_{n} \left[F_{Y^x}(\theta_n(x))- F_{Y^x}(\theta^*(x))  \right]\;. \\
\end{aligned}
\end{equation}
\end{small}
We now use our hypothesis. By \textbf{A1},

\begin{equation*}
\begin{aligned}
 |F_{Y^{B_n^{k_{n+1}}(x)}}(\theta_n(x)) - F_{Y^x}(\theta_n(x))|  \geq   M(x) ||X-x||_{(k_{n+1}, n)}\;,\\
\end{aligned}
\end{equation*}
and by \textbf{A3},

\begin{equation*}
\begin{aligned}
|\theta_n(x)- \theta^*(x)| \leq \sqrt{C_1}\;.
\end{aligned}
\end{equation*}
Thus,
\begin{small}
\begin{equation*}
\begin{aligned}
- 2 \gamma_{n+1} (\theta_n(x)- \theta^*(x))  P_{n} \left[ F_{Y^{B_n^{k_{n+1}}(x)}}(\theta_n(x)) - F_{Y^x}(\theta_n(x)) \right]\\ \leq 2 \gamma_{n+1}\sqrt{C_1}M(x) P_n ||X-x||_{(k_{n+1}, n)}\;.
\end{aligned}
\end{equation*}
\end{small}
On the other hand, thanks to \textbf{A4} we know that,
\begin{equation*}
\begin{aligned}
\left( \theta_n - \theta^*(x) \right) & \left[F_{Y^x}(\theta_n(x))- F_{Y^x}(\theta^*(x)) \right] 
 \geq C_2(x, \alpha) \left[ \theta_n(x) - \theta^*(x) \right]^2 \;. \\
\end{aligned}
\end{equation*}
Coming back to Equation \eqref{hypo}, we get 
\begin{small}
\begin{equation*}
\begin{aligned}
\E_n \left( \theta_{n+1}(x)- \theta^*(x))^2 \right)  & \leq  \left(\theta_{n}(x)- \theta^*(x) \right)^2 \left( \mathbb{1}_{A{n}} + \mathbb{1}_{\bar{A{n}}} \right)  + \gamma_{n+1}^2 P_n\\
& -2 \gamma_{n+1} \left(\theta_{n}(x) - \theta^*(x) \right)^2 C_2(x,  \alpha) P_{n}\\ & + 2 \gamma_{n+1} M(x) \sqrt{C_1} ||X-x||_{(k_{n+1}, n)} P_{n}\;.\\
\end{aligned}
\end{equation*}
\end{small}
To conclude, we take the expectation
\begin{equation*}
\begin{aligned}
b_{n+1}(x) & \leq C_1\P(A_{n}^C) +b_n(x) - 2 \gamma_{n+1} C_2(x, \alpha) \E \left[ P_{n} \left( \theta_n(x) - \theta^*(x) \right)^2 \right] \\
& + \gamma_{n+1}^2 \E(P_{n}) + 2 \gamma_{n+1} \sqrt{C_1} M(x) \E \left[P_{n} ||X-x||_{(k_{n+1}, n)} \right]\;. \\
\end{aligned}
\end{equation*}

But, by definition of $A_n$,we get
\begin{equation*}
\begin{aligned}
- 2 \gamma_{n+1} C_2(x, \alpha) \E \left[ P_{n+1} \left( \theta_n(x) - \theta^*(x) \right)^2 \right] & \leq - \gamma_{n+1} \zeta_{n} C_2(x, \alpha)
 \E \left[ \left(\theta_n(x) - \theta^*(x) \right)^2 \mathbb{1}_{A_{n}} \right] \\
& = - 2 \gamma_{n+1} \zeta_{n} C_2(x, \alpha) b_n(x);. \\
\end{aligned}
\end{equation*}

Finally,
\begin{equation*}
\begin{aligned}
b_{n+1}(x) &\leq b_n(x) \left( 1 - 2 C_2(x, \alpha) \gamma_{n+1}\zeta_{n} \right) + e_{n+1}\;,\\
\end{aligned}
\end{equation*}
with $$e_{n+1}:=C_1\P(A_{n}^C) + \gamma_{n+1}^2 \E(P_{n}) + 2 \gamma_{n+1} \sqrt{C_1} M(x) \E \left[P_{n} ||X-x||_{(k_{n+1}, n)} \right]\;.$$
Now using Lemmas \ref{unif}, \ref{dev} and \ref{Pn} we get for $n \geq N_0$
with \[e_n \leq d_n := C_1 \exp \left( - \frac{3n^{1- \zeta}}{8} \right) +2\sqrt{C_1} M(x) C_3(d) \gamma_{n} \left( \frac{k_n}{n} \right)^{\frac{1}{d}+1} + \gamma_{n}^2\frac{k_n}{n}\;.\]
The conclusion holds thanks to Lemma \ref{Guillaume}, for $n \geq N_0+1$,

\begin{equation}
\begin{aligned}
\label{finalbn}
b_n(x) & \leq \exp \left( -2 C_2(x, \alpha) (\kappa_n - \kappa_{N_0}) \right) b_{N_0}(x)  
 +  \displaystyle\sum_{k=N_0+1}^n \exp \left( -2C_2(x, \alpha) \left( \kappa_n - \kappa_k \right) \right) d_k\;. \\
\end{aligned}
\end{equation}
But thanks to Assumption \textbf{A3}, we have already shown that
$b_{N_0}(x) \leq a_{N_0}(x) \leq C_1$.
To conclude, we re-inject Equation (\ref{finalbn}) in Equation (\ref{debutant}) and obtain for $n \geq N_0+1$,
\begin{equation*}
\begin{aligned}
a_n(x) & \leq \exp \left( -2 C_2(x, \alpha) (\kappa_n- \kappa_{N_0})\right) C_1 
 +  \displaystyle\sum_{k=N_0+1}^{n} \exp \left( -2 C_2(x, \alpha)\left( \kappa_n - \kappa_k \right) \right) d_k \\
 & + C_1 \exp \left( - \frac{3n^{1- \zeta}}{8} \right)\;.\\
\end{aligned}
\end{equation*}

\subsection{Proof of Corollary \ref{premier} : Rate of convergence}

In this part, we will denote $$T^0_n:=C_1 \exp \left( \frac{-3n^{1-\zeta}}{8} \right)\;, \, \, T^1_n:=\exp \left( -2 C_2(x, \alpha) (\kappa_n - \kappa_{N_0}) \right)\;$$ and $$T^ 2_n:=\displaystyle\sum_{k=N_0+1}^{n} \exp\left( -2 C_2(x, \alpha)\left( \kappa_n - \kappa_k \right) \right) d_k\;.$$ We want to find a simpler expression for those terms to better see their order in $n$. First, considering $T^1_n$ we see that $a_n(x)$ can converge to 0 only when the sum $$\displaystyle\sum_{k \geq 1} \frac{1}{k^{\gamma+ \zeta}}=+ \infty.$$ This is why we must first consider $\zeta \leq 1 - \gamma$. As $\zeta < 1 - \beta$, we have to take $\beta > \gamma$. 

\begin{remark}
The frontier case $\zeta=1- \gamma$ is possible but the analysis shows that it is a less interesting choice than $\zeta < 1- \gamma$ (there is a dependency in the value of $C_2(x, \alpha)$ but the optimal rate is the same as the one in the case we study). In the sequel, we only consider $\zeta < 1 - \gamma$. 
\end{remark}

Let us upper-bound $T_n^1$. As $x \mapsto 1/x^{\zeta+\gamma}$ is decreasing, we get 
\begin{equation*}
\begin{aligned}
T_n^1& = \exp \left( -2C_2(x, \alpha) \displaystyle\sum_{k=N_0+1}^n \frac{1}{k^{\zeta + \gamma}} \right) \\
& \leq \exp \left( -2C_2(x, \alpha) \int_{N_0+1}^{n+1} \frac{1}{t^{\zeta + \gamma}} dt \right)\\
& \leq \exp \left( -2C_2(x, \alpha) \frac{(n+1)^{1-\zeta-\gamma}-(N_0+1)^{1- \zeta - \gamma}}{(1- \zeta - \gamma)} \right)\;.\\
\end{aligned}
\end{equation*}
Then, $T_n^1$ (just like $T_n^0$) is exponentially small when $n$ grows up. To deal with the second term $T_n^2$ we first study the order in $n$ of $d_n$. $d_n$ is composed of three terms :
$$d_n \leq C_1 \exp \left( - \frac{3n^{1- \zeta}}{8} \right) +2\sqrt{C_1} M(x) C_3(d) n^{-\gamma + (\beta-1)(1+ \frac{1}{d})} +n^{-2\gamma + \beta-1}\;.$$

The first one is negligeable (exponentially decreasing). Let us compare the two others which are powers of $n$. Comparing their exponents, we get that there exists constants $C_5$ and $C_6(d)$ (their explicit form is given in the Appendix) such that

\begin{itemize}

\item[•] if $\beta \leq 1 - d\gamma$, then for $n \geq N_0+1$,
\[d_n \leq C_5(x, d)n^{-2\gamma+\beta-1}\;,\]

\item[•] if $\beta > 1 - d\gamma$, then for $n \geq N_0+1$,
\[d_n \leq C_6(x,d) n^{-\gamma+(1+ \frac{1}{d})(\beta-1)}\;.\]
\end{itemize}

\begin{remark}
Let us detail how one can find $C_5$ (it is the same reasoning for $C_6$). If $\beta \leq 1 - d \gamma$, we know that when $n$ will be big enough, the dominating term of $d_n$ will be the one in $n^{-2\gamma + \beta-1}$. Then, it is logical to search a constant $C_5(x, d)$ such that $\forall n \geq N_0+1$,

$$d_n \leq \frac{C_5(x, d)}{n^{2\gamma - \beta +1}}\;.$$
Such a constant has to satisfy, for all $n \geq N_0+1$, 

$$C_5(x, d) \geq C_1 \exp \left(-\frac{3}{8}n^{1-\zeta} \right)n^{2\gamma - \beta+1} + \frac{2\sqrt{C_1}M(x)C_3(d)}{n^{-\gamma + (1- \beta)/d}} +1\;.$$
Since $\beta \leq 1 - d \gamma$, the map $x \mapsto  \frac{2\sqrt{C_1}M(x)C_3(d)}{n^{-\gamma + (1- \beta)/d}}$ is positive and decreasing. Then its maximum is reached for $n=N_0+1$. Moreover, the map $x \mapsto C_1 \exp \left(-\frac{3}{8}n^{1-\zeta} \right)n^{2\gamma - \beta+1}$ is also positive and is decreasing on an $[A, + \infty[$. It also has a maximum. The previous inequality is then true for 

$$C_5(x, d) := \max_{n \geq N_0+1}C_1 \exp \left(-\frac{3}{8}n^{1-\zeta} \right)n^{2\gamma - \beta+1} +\frac{2\sqrt{C_1}M(x)C_3(d)}{(N_0+1)^{-\gamma + (1- \beta)/d}}+1\;.$$
\end{remark}

Let us study the two previous cases.

\vspace{0.3cm}
\textbf{Study of $T_n^2$ when $\beta > 1 - d\gamma$ :}
\vspace{0.3cm}

To upper-bound these sums, we use arguments from \cite{cenac}, which studies the stochastic algorithm to estimate the median on an Hilbert space. The main arguments are comparisons between sums and integrals. Indeed, for $ n \geq N_0+2$ and $n \geq N_3$ where $N_3$ is such that

$$\forall n \geq N_3, \, \, \lfloor \frac{n}{2} \rfloor \geq N_0+1\;,$$

\begin{small}
\begin{equation*}
\begin{aligned}
T_n^2 & = C_6(x,d)\displaystyle\sum_{k=N_0+1}^{n-1} \exp \left( -2C_2(x, \alpha) \displaystyle\sum_{j=k+1}^n \frac{a}{j^{\zeta+ \gamma}} \right) \frac{1}{k^{\gamma +\left( 1 + \frac{1}{d} \right)(1- \beta)}} + \frac{C_6(x, d)}{n^{\gamma + (1 + \frac{1}{d})(1- \beta)}} \\
& = C_6(x,d)\displaystyle\sum_{k=N_0+1}^{\lfloor \frac{n}{2} \rfloor } \exp \left( -2C_2(x, \alpha) \displaystyle\sum_{j=k+1}^n \frac{a}{j^{\zeta+ \gamma}} \right) \frac{1}{k^{\gamma +\left( 1 + \frac{1}{d} \right)(1- \beta)}} \\
& + C_6(x,d) \displaystyle\sum_{k=\lfloor \frac{n}{2} \rfloor +1}^{n-1} \exp \left( -2C_2(x, \alpha)\displaystyle\sum_{j=k+1}^{n} \frac{a}{j^{\zeta+ \gamma}} \right) \frac{1}{k^{\gamma +\left( 1 + \frac{1}{d} \right)(1- \beta)}} + \frac{C_6(x, d)}{n^{\gamma + (1 + \frac{1}{d})(1- \beta)}} \\
& =: S_1 + S_2 + S_3\;. \\
\end{aligned}
\end{equation*}
\end{small}
First, the function $x \mapsto x^{-\zeta - \gamma}$ is decreasing on $]0, + \infty[$ then

\begin{equation*}
\begin{aligned}
S_2 &\leq C_6(x,d) \displaystyle\sum_{k=\lfloor \frac{n}{2} \rfloor+1}^{n-1} \exp \left( -2C_2(x, \alpha) \int_{k+1}^{n+1} \frac{1}{x^{\zeta+ \gamma}} dx \right) \frac{1}{k^{\gamma + (1 + \frac{1}{d})(1- \beta)}}\\
&= C_6(x,d) \exp \left(-2C_2(x, \alpha)\frac{(n+1)^{1- \gamma - \zeta}}{1- \gamma - \zeta} \right) \\
& \hspace{4cm}  \displaystyle\sum_{k=\lfloor \frac{n}{2} \rfloor+1}^{n-1} \exp \left( -2C_2(x, \alpha) \frac{(k+1)^{1- \gamma - \zeta}}{1- \gamma - \zeta}\right)\frac{1}{k^{\gamma + (1 + \frac{1}{d})(1- \beta)}}\;.\\
\end{aligned}
\end{equation*}
Then, taking, $1- \beta< \zeta < \min((1- d\gamma), \left( 1 + \frac{1}{d} \right) (1- \beta)) $, we have since $k \geq \lfloor \frac{n}{2} \rfloor +1$

\begin{multline*}
S_2  \leq C_6(x,d) \exp \left(-2C_2(x, \alpha) \frac{(n+1)^{1- \gamma - \zeta}}{1- \gamma - \zeta} \right) \left( \frac{2}{n} \right)^{(1+ \frac{1}{d})(1- \beta)- \zeta}\\ \displaystyle\sum_{k=\lfloor \frac{n}{2} \rfloor+1}^{n-1} \exp \left( -2C_2(x, \alpha) \frac{(k+1)^{1- \gamma - \zeta}}{1- \gamma - \zeta}\right)\frac{1}{k^{\gamma +\zeta}}\;.
 \end{multline*}
Now, since for $k \geq 1$, 

$$\left( \frac{1}{k} \right)^{\zeta+\gamma} \leq \left( \frac{2}{k+1} \right)^{\zeta+\gamma}\;,$$ we get
\begin{small}
\begin{equation*}
\begin{aligned}
S_2 & \leq C_6(x,d) \exp \left(-2C_2(x, \alpha) \frac{(n+1)^{1- \gamma - \zeta}}{1- \gamma - \zeta} \right) \left( \frac{2}{n} \right)^{(1+ \frac{1}{d})(1- \beta)- \zeta} 2^{\zeta+ \gamma}\\
& \hspace{4cm} \displaystyle\sum_{k=\lfloor \frac{n}{2} \rfloor+1}^{n-1} \exp \left( -2C_2(x, \alpha) \frac{(k+1)^{1- \gamma - \zeta}}{1- \gamma - \zeta}\right)\frac{1}{(k+1)^{\gamma + \zeta)}}\;.\\
\end{aligned}
\end{equation*}
\end{small}
Since the function $x \mapsto \exp \left( 2 C_2(x, \alpha) \frac{n^{1 - \zeta - \gamma}}{1- \zeta -\gamma} \right)$ is decreasing on $\left[\frac{2C_2(x, \alpha)}{\gamma+ \zeta}, + \infty \right[$, we also define the integer $N_1(x, \alpha)$ the rank such that

$$\forall n \geq N_1(x, \alpha), \, \, \lfloor \frac{n}{2} \rfloor + 1 \geq \frac{2C_2(x, \alpha)}{\zeta + \gamma}\;.$$
For $n \geq N_1(x, \alpha)$ we get

\begin{small}
\begin{equation*}
\begin{aligned}
S_2 & \leq C_6(x,d) \exp \left(-2C_2(x, \alpha) \frac{(n+1)^{1- \gamma - \zeta}}{1- \gamma - \zeta} \right)  \frac{2^{(1+ \frac{1}{d})(1- \beta) + \gamma}}{n^{(1+ \frac{1}{d})(1- \beta)- \zeta}} \\
& \times \displaystyle\sum_{k=\lfloor \frac{n}{2} \rfloor+1}^{n-1} \int_{\lfloor \frac{n}{2} \rfloor +2}^{n}\exp \left( -2C_2(x, \alpha) \frac{x^{1- \gamma - \zeta}}{1- \gamma - \zeta}\right)\frac{1}{x^{\gamma+ \zeta}}dx\\
&  \leq \frac{C_6(x,d)}{2C_2(x, \alpha)} \exp \left(-2C_2(x, \alpha) \frac{(n+1)^{1- \gamma - \zeta}}{1- \gamma - \zeta} \right) \frac{2^{(1+ \frac{1}{d})(1- \beta) + \gamma}}{n^{(1+ \frac{1}{d})(1- \beta)- \zeta}} \\
& \times \left[ \exp \left( 2C_2(x, \alpha) \frac{n^{1- \zeta- \gamma}}{1- \zeta - \gamma} \right) - \exp \left( 2C_2(x, \alpha) \frac{(\lfloor \frac{n}{2} \rfloor +2)^{1- \zeta- \gamma}}{1- \zeta - \gamma} \right)\right]\\
\leq & \frac{C_6(x,d)}{2C_2(x, \alpha)}\frac{2^{(1+ \frac{1}{d})(1- \beta) + \gamma}}{n^{(1+ \frac{1}{d})(1- \beta)- \zeta}}=: \frac{C_7(x, d, \alpha)}{2} \frac{1}{n^{-\zeta + (1+ \frac{1}{d})(1-\beta)}}\;.\\
\end{aligned}
\end{equation*}
\end{small}
Let us now deal with the term $S_1$. As $k \leq  \lfloor \frac{n}{2} \rfloor$, we have

\[\displaystyle\sum_{j=k+1}^{n} \frac{1}{j^{\zeta+\gamma}} \geq \frac{n}{2} \frac{1}{n^{\zeta+\gamma}}\;.\]
Then,
\begin{small}
\begin{equation*}
\begin{aligned}
S_1 & = C_6(x,d) \displaystyle\sum_{k=N_0+1}^{\lfloor \frac{n}{2} \rfloor} \exp \left( -2C_2(x, \alpha) \displaystyle\sum_{j=k+1}^n \frac{a}{j^{\zeta+ \gamma}} \right) \frac{1}{k^{\gamma + (1- \beta) \left( 1 + \frac{1}{d} \right)}} \\
& \leq C_6(x,d)\displaystyle\sum_{k=1}^{\lfloor \frac{n}{2} \rfloor} \exp \left( -C_2(x, \alpha)n^{1- \zeta - \gamma} \right) \frac{1}{k^{\gamma + (1- \beta) \left( 1 + \frac{1}{d} \right)}}\\
& \leq C_6(x,d) \exp \left( -C_2(x, \alpha) n^{1-\zeta- \gamma} \right) \displaystyle\sum_{k=1}^{\lfloor \frac{n}{2} \rfloor} \frac{1}{k^{\gamma + (1- \beta) \left( 1 + \frac{1}{d} \right)}}\;.\\
\end{aligned}
\end{equation*}
\end{small}
Thanks to the exponential term, $S_1$ is insignificant compared to $S_2$ whatever is the behaviour of the sum $\displaystyle\sum_k k^{-\gamma - (1- \beta) \left( 1 + \frac{1}{d} \right)}$, and so is $T_1^n$. Then, denoting $N_2(d, x)$ the rank after which we have

\[S_3+S_1+ T_n^1+T_n^0 \leq \frac{C_{7}(x, \alpha, d)}{2n^{(1+\frac{1}{d})(1-\beta)-\zeta}}\;,\]
we get, in the case where $\beta > 1 - \gamma$ and $1- \beta< \zeta < \min((1- \gamma), \left( 1+ \frac{1}{d} \right) (1- \beta)) $, for $n \geq \max\left(N_0, N_1(x,\alpha), N_2(d,x)\right)$

\begin{equation*}
\begin{aligned}
a_n(x) & \leq \frac{C_7(x, \alpha, d)}{n^{-\zeta + \left( 1+ \frac{1}{d} \right) (1- \beta)}}\;.\\
\end{aligned}
\end{equation*}

\vspace{0.3cm}

\textbf{Study of $T_n^2$ when $\beta \leq 1 - d\gamma$ :}

\vspace{0.3cm}

Using the same arguments, we conclude that for $1- \beta < \zeta < \min(1- \beta + \gamma, 1- \gamma)$ and $n \geq \max (N_0, N_1(x,\alpha),N_2(d, x))$ (see Appendix for precise definitions of these ranks), there exists a constant $C_{8}(x, \alpha, d)$ such that the mean squared error satisfies

\[a_n(x) \leq \frac{C_{8}(x, \alpha, d)}{n^{\gamma- \beta +1 -\zeta}}\;.\]
\subsection{Proof of Corollary \ref{doubleoptim} : choice of best parameters $\beta$ and $\gamma$}

Let us now optimize the rate of convergence obtained in previous theorem. When $\beta\geq \gamma$ and $\beta \leq 1-d\gamma$, the rate of convergence is of order $n^{-\gamma+\beta-1+\zeta}$. To optimize it, we have to choose $\zeta$ as small as possible. Then, we take $\zeta=1-\beta+\eta_{\zeta}$. The rate becomes $n^{-\gamma+\eta_{\zeta}}$. Then, we have also to choose $\gamma$ as small as possible. In this area, there is only one point in which $\gamma$ is the smallest, this is the point $(\gamma, \beta)=(\frac{1}{1+d}, \frac{1}{1+d})$. Since we have to take $\beta > \gamma$, the best couple of parameters, in this area, is $(\frac{1}{1+d}, \frac{1}{1+d} + \eta_{\beta})$. These parameters follow a rate of convergence of $n^{\frac{-1}{1+d} + \eta}$. 

When we are in the second area, the same kind of arguments allows us to conclude to the same optimal point with the same rate of convergence. 

In Figure \ref{fig:finale}, we use the numerical simulations of Section 3 to illustrate the previous discussion. 

\begin{figure}[ht!]
$$\includegraphics[scale=1]{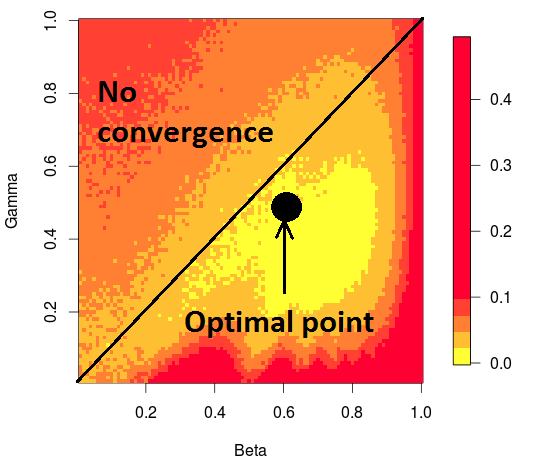}$$
\caption{Theoretical behaviour of the MSE in function of $\beta$ and $\gamma$, $n=200$.}\label{fig:finale}
\end{figure}

We have finally shown that

\[a_n(x) \leq \frac{C_{9}(x, \alpha, d)}{n^{\frac{1}{1+d} - \eta}}\; ,\]
where the constant is the minimal constant between $C_7(x, \alpha, d)$ and $C_8(x, \alpha, d)$ computed with optimal parameters $(\gamma, \beta, \zeta)$.

\newpage

\section{Recap of the constants}
\label{Apx2}
Let us sum up all the constants we need in this paper.

\subsection{Constants of the model}
We denote :
\begin{itemize}
\item[$\bullet$] $M(x)$ the constant of continuity of the model, that is

$$\forall B \in \mathcal{B}_x, \, \, \forall t \in \mathbb{R}, \, \, \left| F_{Y^ B}(t)-F_{Y^x}t) \right| \leq M(x) r_B\;.$$

\item[$\bullet$] $C_{input}$ is the positive lower bound of the density of the inputs law $f_X$. 

\item[$\bullet$] $C_g(x)$ is the positive lower bound of the density of the law of $g(x, \varepsilon)$. 
\end{itemize}

\subsection{Compact support}
We denote :

\begin{itemize}
\item[$\bullet$] $[L_Y, U_Y]$ the compact in which are included the values of $g$. 
\item[$\bullet$] $[L_X, U_X]$ the compact in which is included the support of the distribution of $X$. 
\item[$\bullet$] $[L_{\theta_n}, U_{\theta_n}]:=[L_Y-(1-\alpha), U_Y+\alpha]$ the segment in which $\theta_n$ can take its values ($\forall x$).
\item[$\bullet$] $U_{|.|}$ the upper bound of the compact support of the distribution of $||X-x||$ ($\forall x)$. 
\end{itemize}

\subsection{Real constants}

We denote :

\begin{itemize}
\item[$\bullet$] $\sqrt{C_1}:=U_Y + \alpha -L_Y .$ $C_1$ is the uniform in $\omega$ and $x$ bound of $\left(\theta_n(x) - \theta^*(x) \right)^2$.

\item[$\bullet$] $C_2(x, \alpha):= \min \left(C_g(x), \frac{1-\alpha}{U_Y+\alpha-L_Y}\right)$ is the constant such that 

$$\left[ F_{Y^x} \left( \theta_n(x) \right) - F_{Y^x} \left( \theta^*(x) \right) \right] \left[ \theta_n(x) - \theta^*(x) \right] \geq C_2(x, \alpha) \left( \theta_n(x) - \theta^*(x) \right)^2\;.$$

\item[$\bullet$] $C_3(d):=\sqrt[d]{2} \left( 1+\frac{8}{3d}+ \frac{1}{\sqrt[d]{C_{input}}C_4(d)} \right)$.
\item[$\bullet$] $C_4(d):=\frac{\pi^{\frac{d}{2}}}{\Gamma \left( \frac{d}{2} \right)+1}$.
\item[$\bullet$]$C_5(x, d):=\displaystyle\max_{n \geq N_0+1}C_1 \exp \left(-\frac{3}{8}n^{1-\zeta} \right)n^{2\gamma - \beta+1} +\frac{2\sqrt{C_1}M(x)C_3(d)}{(N_0+1)^{-\gamma + (1- \beta)/d}}+1$.
\item[$\bullet$] $\begin{aligned}
        C_6(x, d):= & \displaystyle\max_{n \geq N_0+1}C_1 \exp \left( - \frac{3}{8}n^{1- \zeta}\right) n^{\gamma +(1+ \frac{1}{d})(1-\beta)}\\ & + 2 \sqrt{C_1}M(x) C_3(d)+ \frac{1}{(N_0+1)^{\gamma - \frac{1}{d}(1-\beta)}}.
\end{aligned}$
\item[$\bullet$] \begin{small} $\begin{aligned}
        C_5^{optim}:=& \displaystyle\max_{n \geq N_0+1}C_1 \exp \left(-\frac{3}{8}n^{\left( \frac{1}{1+d}+\eta_{\beta}\right)-\eta_{\zeta}} \right)(N_0+1)^{\frac{1}{1+d}-\eta_{\beta}+1} \\ & +1 + \frac{1}{(N_0+1)^{-\frac{1}{1+d} + \frac{1}{d}(1- \frac{1}{1+d}-\eta_{\beta})}}
\end{aligned}$.
\end{small}

\item[$\bullet$] \begin{small} $\begin{aligned}
        C_6^{optim}(x, d):= & \displaystyle\max_{n \geq N_0+1}C_1 \exp \left( - \frac{3}{8}n^{\left( \frac{1}{1+d}+\eta_{\beta}\right)-\eta_{\zeta}}\right) n^{\left(1+\frac{1}{d}\right)-\frac{1}{d(1+d)}-\eta_{\beta}\left( 1+\frac{1}{d}\right)} \\ & + 2 \sqrt{C_1}M(x) C_3(d)+ \frac{1}{(N_0+1)^{-\frac{1}{d}+\frac{1}{d(1+d)}+ \frac{1}{1+d}+\frac{\eta_{\beta}}{d}}}
\end{aligned}$.
\end{small}

\item[$\bullet$] $C_7(x, \alpha, d):= \frac{2^{(1+ \frac{1}{d})(1-\beta) + \gamma}C_6(x,d)}{C_2(x, \alpha)}.$
\item[$\bullet$] $ C_8(x, \alpha):= \frac{2^{2\gamma - \beta+1}C_5(x, d)}{C_2(x, \alpha)}.$
\item[$\bullet$] $C_9(x, \alpha, d):= \min \left( \frac{2^{1+ \frac{1}{d} - \frac{1}{d(1+d)}-\eta_{\beta}(1+ \frac{1}{d})}C_5^{optim}(x, d)}{C_2(x, \alpha)}, \frac{2^{\frac{1}{1+d}-\eta_{\beta}+1}C_6^{optim}(x, d)}{C_2(x, \alpha)} \right).$
\item[$\bullet$] $C_{10}(d):= \sqrt[d]{\frac{2(k_n+1)}{(n+1) C_{input} C_4(d)}}.$
\end{itemize}

\subsection{Integer constants}

We denote :

\begin{itemize}
\item[$\bullet$] $N_0:= 2^{\frac{1}{\zeta-(1-\beta)}}.$
%
%
%
%
%
%
%
%

\item[$\bullet$] $N_1(x, \alpha)$ is the rank such that $n \geq N_1(x, \alpha)$ implies

$$\lfloor \frac{n}{2} \rfloor + 1 \geq \frac{2C_2(x, \alpha)}{\zeta + \gamma}.$$

\item[$\bullet$] $N_2(x, \alpha, d)$ is the integer such that $\forall n \geq N_2(x, \alpha, d)$, 

\begin{enumerate}
\item[a)] If $\beta \leq 1-d\gamma$,

$$ S_3+S_1 + T_n^1 +T_n^0 \leq \frac{C_7(x, \alpha, d)}{2n^{\left(1+\frac{1}{d}\right)(1-\beta)-\zeta}}\;,$$

where $T_n^1:=\exp \left( -2C_2(x, \alpha) \displaystyle\sum_{k=N_0+1}^{n} k^{-\gamma-\zeta}\right) $, $T_n^0:=C_1 \exp \left( \frac{-3n^{1-\zeta}}{8}\right)$, 
\newline
$S_3:=\frac{C_6(x, d)}{n^{\gamma + (1+ \frac{1}{d})(1- \beta)}}$ and \\ $S_1:= C_6(x, d) \exp(-2C_2(x, \alpha)n^{1-\zeta-\gamma}) \displaystyle\sum_{k=1}^{\lfloor \frac{n}{2} \rfloor} k^{-\gamma-(1-\beta)(1+1/d)}$. 

\item[b)] If $\beta >1-d\gamma$,

$$S_3+ S_1+ T_n^1 +T_n^0 \leq \frac{C_8(x, \alpha, d)}{2n^{\gamma-\beta+1-\zeta}}\;,$$

where $T_n^1:=\exp \left( -2C_2(x, \alpha) \displaystyle\sum_{k=N_0+1}^{n} k^{-\gamma-\zeta}\right) $, $T_n^0:=C_1 \exp \left( \frac{-3n^{1-\zeta}}{8}\right)$, 
\newline

$S_3:=\frac{C_5}{n^{2\gamma - \beta+1)}}$ and $S_1:= C_5 \exp(-2C_2(x, \alpha)n^{1-\zeta-\gamma}) \displaystyle\sum_{k=1}^{\lfloor \frac{n}{2} \rfloor} k^{-\gamma-(1-\beta)(1+1/d)}$. 
\end{enumerate}

\item[$\bullet$] $N_3$ is the rank such that $\forall n \geq N_3, \, \, \lfloor \frac{n}{2} \rfloor \geq N_0+1.$

\item[$\bullet$] $N_4(x, \alpha, d) :=\max\left(N_0+2, N_1(x, \alpha), N_2(x, \alpha, d), N_3\right)$.

\end{itemize}

%


\bibliographystyle{spmpsci}      
\bibliography{bib1.bib}   

\begin{thebibliography}{10}
\providecommand{\url}[1]{{#1}}
\providecommand{\urlprefix}{URL }
\expandafter\ifx\csname urlstyle\endcsname\relax
  \providecommand{\doi}[1]{DOI~\discretionary{}{}{}#1}\else
  \providecommand{\doi}{DOI~\discretionary{}{}{}\begingroup
  \urlstyle{rm}\Url}\fi

\bibitem{Abra}
Abramowitz, M., Stegun, I.A.: Handbook of Mathematical Functions.
\newblock Dover Publications (1965)

\bibitem{bect}
Bect, J., Ginsbourger, D., Li, L., Picheny, V., Vazquez, E.: Sequential design
  of computer experiments for the estimation of a probability of failure.
\newblock Statistics and Computing \textbf{22}(3), 773--793 (2012)

\bibitem{bata2}
Bhattacharya, P.K., Gangopadhyay, A.K.: Kernel and nearest-neighbor estimation
  of a conditional quantile.
\newblock The Annals of Statistics \textbf{18}(3), 1400--1415 (1990)

\bibitem{blum}
Blum, J.R.: Approximation methods which converge with probability one.
\newblock The Annals of Mathematical Statistics \textbf{25}(2), 382--386 (1954)

\bibitem{cenac}
Cardot, H., C{\'e}nac, P., Godichon, A.: Online estimation of the geometric
  median in hilbert spaces: non asymptotic confidence balls.
\newblock The Annals of Statistics \textbf{45}(2), 591--614 (2017)

\bibitem{order}
David, H.A., Nagaraja, H.N.: Order Statistics.
\newblock Wiley (2003)

\bibitem{Devroye}
Devroye, L., Gy{\"o}rfi, L., Lugosi, G.: A probabilistic theory of pattern
  recognition, vol.~31.
\newblock Springer Science \& Business Media (2013)

\bibitem{Duflo}
Duflo, M.: Random Iterative Models, 1st edn.
\newblock Springer-Verlag, Berlin, Heidelberg (1997)

\bibitem{fab}
Fabian, V.: On asymptotic normality in stochastic approximation.
\newblock The Annals of Mathematical Statistics \textbf{39}(4), 1327--1332
  (1968)

\bibitem{meno}
Frikha, N., Menozzi, S.: Concentration bounds for stochastic approximations.
\newblock Electron. Commun. Probab \textbf{17}(47), 1--15 (2012)

\bibitem{gadat}
Gadat, S., Klein, T., Marteau, C.: Classification with the nearest neighbor
  rule in general finite dimensional spaces: necessary and sufficient
  conditions.
\newblock The Annals of Statistics \textbf{44}(3), 982--1009 (2016)

\bibitem{godichon}
Godichon, A.: Estimating the geometric median in hilbert spaces with stochastic
  gradient algorithms.
\newblock Journal of Multivariate Analysis \textbf{146}, 209--222 (2016)

\bibitem{jala2}
Jala, M., L{\'e}vy-Leduc, C., Moulines, {\'E}., Conil, E., Wiart, J.:
  Sequential design of computer experiments for the assessment of fetus
  exposure to electromagnetic fields.
\newblock Technometrics \textbf{58}(1), 30--42 (2014)

\bibitem{code1}
Kennedy, M.C., O'Hagan, A.: Predicting the output from a complex computer code
  when fast approximations are available.
\newblock Biometrika \textbf{87}(1), 1--13 (2000)

\bibitem{moul}
Moulines, E., Bach, F.R.: Non-asymptotic analysis of stochastic approximation
  algorithms for machine learning.
\newblock In: Advances in Neural Information Processing Systems, pp. 451--459
  (2011)

\bibitem{oakley}
Oakley, J.: Estimating percentiles of uncertain computer code outputs.
\newblock Journal of the Royal Statistical Society: Series C (Applied
  Statistics) \textbf{53}(1), 83--93 (2004)

\bibitem{rob}
Robbins, H., Monro, S.: A stochastic approximation method.
\newblock The Annals of Mathematical Statistics \textbf{22}(3), 400--407 (1951)

\bibitem{rup}
Ruppert, D.: Handbook of sequential analysis.
\newblock CRC Press (1991)

\bibitem{sac}
Sacks, J.: Asymptotic distribution of stochastic approximation procedures.
\newblock The Annals of Mathematical Statistics \textbf{29}(2), 373--405 (1958)

\bibitem{code3}
Sacks, J., Welch, W.J., Mitchell, T.J., Wynn, H.P.: Design and analysis of
  computer experiments.
\newblock Statistical science \textbf{4}(4), 409--423 (1989)

\bibitem{code2}
Santner, T.J., Williams, B.J., Notz, W.I.: The design and analysis of computer
  experiments.
\newblock Springer Science \& Business Media (2013)

\bibitem{schrek}
Schreck, A., Fort, G., Moulines, E., Vihola, M.: {Convergence of Markovian
  Stochastic Approximation with discontinuous dynamics}.
\newblock SIAM J. Control Optim. \textbf{54}(2), 866--893 (2016)

\bibitem{sto1}
Stone, C.J.: Nearest neighbour estimators of a nonlinear regression function.
\newblock Proc. Comp. Sci. Statis. 8th Annual Symposium on the Interface pp.
  413--418 (1976)

\bibitem{sto2}
Stone, C.J.: Consistent nonparametric regression.
\newblock The Annals of Statistics \textbf{5}(4), 595--620 (1977)

\bibitem{woodroofe}
Woodroofe, M.: {Normal approximation and large deviations for the Robbins-Monro
  process}.
\newblock Probability Theory and Related Fields \textbf{21}(4), 329--338 (1972)

\end{thebibliography}


\end{document}